\documentclass[12pt]{article}

\usepackage{amsfonts,amssymb,latexsym,amsmath}

\usepackage[english]{babel}
\usepackage[margin=1in]{geometry}

\usepackage{enumitem}
\usepackage{verbatim}

\newtheorem{thm}{Theorem}[section]
\newtheorem{prop}[thm]{Proposition}
\newtheorem{lem}[thm]{Lemma}
\newtheorem{cor}[thm]{Corollary}

\newtheorem{example}[thm]{Example}
\newtheorem{conj}[thm]{Conjecture}

\newenvironment{proof}
{\par\addvspace{0.3cm}\noindent{\rm Proof. }}
{\nopagebreak\mbox{}\hfill $\Box$\par\addvspace{0.25cm}}

\newcommand{\T}{\mathbb{T}}
\newcommand{\Z}{\mathbb{Z}}
\newcommand{\R}{\mathbb{R}}
\newcommand{\C}{\mathbb{C}}
\newcommand{\Q}{\mathbb{Q}}

\newcommand{\Zp}{\Z^+}

\newcommand{\cR}{\mathcal{R}}
\newcommand{\cC}{\mathcal{C}}
\newcommand{\cA}{\mathcal{A}}
\newcommand{\cU}{\mathcal{U}}
\newcommand{\cL}{\mathcal{L}}

\newcommand{\tp}{{\tau'}}

\newcommand{\HomXT}{\mathrm{Hom}(\Xi, \mathbb{T})}

\newcommand{\iy}{\infty}

\DeclareMathOperator{\gr}{gr}
\DeclareMathOperator{\trace}{trace}

\numberwithin{equation}{section}

\begin{document}

\title{Asymptotics of determinants for finite sections of operators with almost periodic diagonals}
\author{Torsten Ehrhardt\thanks{tehrhard@ucsc.edu}\\
   Department of Mathematics\\
        University of California\\
       Santa Cruz, CA 95064, USA
  \and   Zheng Zhou\thanks{zzho18@ucsc.edu}\\
 Department of Mathematics\\
        University of California\\
       Santa Cruz, CA 95064, USA}

\date{}

\maketitle

\vspace*{-3ex}
\begin{abstract}
Let $A=(a_{j,k})_{j,k=-\infty}^\infty$ be a bounded linear operator on $l^2(\Z)$ whose diagonals 
$D_n(A)=(a_{j,j-n})_{j=-\infty}^\infty \in l^\infty(\Z)$ are almost periodic sequences. For certain classes of such operators and under certain conditions, we are going to determine the asymptotics of the determinants $\det A_{n_1,n_2}$ of the finite sections
$A_{n_1,n_2}=(a_{j,k})_{j,k=n_1}^{n_2-1}$ as their size $n_2-n_1$ tends to infinity.
Examples of such operators include block Toeplitz operators and the almost Mathieu operator.
\end{abstract}


\linebreak
determinants of finite sections


\section{\Large \bf Introduction}
\label{sec:1}

{\bf Asymptotics of determinants of finite sections.}\ 
For a symbol $a\in L^\infty(\T)$, $\T=\{t\in\C\,:\, |t|=1\}$ being the unit circle, the $n\times n$ Toeplitz matrices are defined by
\begin{equation}
T_n(a) = (a_{j-k}),\qquad 0\le  j,k\le n-1,
\end{equation}
where $a_k$ stands for the $k$-th Fourier coefficients of $a$,
\[
a_k = \frac{1}{2\pi} \int_0^{2\pi}a(e^{ix})e^{-ikx}dx,\quad k\in \Z.
\]
Under certain assumption on the symbol $a$, the strong Szeg\"o-Widom limit theorem \cite{szego1952,Widom} states that 
\[
\lim_{n\to\infty}\frac{\det T_n(a)}{G[a]^n} = E[a],
\]
where $G[a]$ and $E[a]$ are  explicitly defined non-zero constants.

The Toeplitz matrices $T_n(a)$ can be viewed as the finite sections of Laurent operators $L(a)$, 
i.e., as the compressions
$$
T_n(a)= P_n L(a) P_n|_{\mathrm{im} P_n},
$$
where the Laurent operators
\begin{equation}\label{L(a)}
L(a) = (a_{j-k}), \qquad j,k\in \Z,
\end{equation}
are acting on $l^2(\Z)$, and $P_n$ are the finite section projections
$$
P_n: \{x_k\}_{k\in \Z} \mapsto \{y_k\}_{k\in \Z}, \quad
        y_k = 
        \begin{cases}
        x_k &\mbox{ if } \enskip 0\le k<n\\ 
        0     &\mbox{ if } \enskip k<0  \mbox{ or }  k\ge n.
        \end{cases}
$$
Note that Laurent operators are constant on each diagonal, that means they are shift-invariant in the sense that 
$U_{-n}L(a)U_n = L(a)$ for all $n\in \Z$, where $U_n=L(t^n)$ are the shift operators
    \begin{equation}\label{Un on l2}
        U_n:\{x_k\}_{k\in \Z} \mapsto \{x_{k-n}\}_{k\in \Z}
    \end{equation}
acting on $l^2(\Z)$.

For $n_1, n_2\in \Z$ such that $n_2>n_1$ let us define more general finite section  projections 
\begin{equation}\label{Pn1n2}
P_{n_1,n_2}:
 \{x_k\}_{k\in \Z} \mapsto \{y_k\}_{k\in \Z}, \enskip
 y_k = \begin{cases}
                x_k &\mathrm{ if } \enskip n_1\le k<n_2\\
                0   &\mathrm{ if } \enskip k<n_1 \enskip \mathrm{ or }\enskip  k\ge n_2
              \end{cases}
\end{equation}
thought of as operators acting on $l^2(\Z)$. 
The goal of the paper is to describe the asymptotic behavior of the determinants
$$
\det P_{n_1,n_2} A P_{n_1,n_2}
$$
as the size $n_2-n_1$ of the finite sections tends to $ \infty$, where $A$ is a bounded linear operator on $l^2(\Z)$ with almost periodic diagonals which has to satisfy various conditions. 
As usual, the finite sections $P_{n_1,n_2} A P_{n_1,n_2}$ therein are identified  with  matrices of order $n_2-n_1$.
Notice, however, that even in the case of band operators  with periodic diagonals (see the comments on Theorem \ref{t1.1} below)
one  {\em cannot} expect an asymptotic behavior of the kind 
\begin{equation}\label{no-limit}
\lim_{n_2-n_1\to \infty}
\frac{ \det P_{n_1,n_2} A P_{n_1,n_2}}{G^{n_2-n_1}} = E.
\end{equation}
In fact, in order for a limit to exist one has to consider appropriate sequences of integers  $h_1=\{h_1(n)\}_{n\ge 0} $ and $h_2=\{h_2(n)\}_{n\ge 0}$ (which will be called  {\em fractal sequences} in this paper) such that 
 $h(n):=h_2(n) - h_1(n)>0$ and $h(n)\to \infty$ as $n\to \infty$. Then, under the appropriate assumptions on  $A$, 
 one of our main results (referred to as the {\em fractal version of the limit theorem})  states that 
\begin{equation}\label{Sz-limit}
\lim_{n\to \infty}
\frac{ \det P_{h_1(n),h_2(n)} A P_{h_1(n),h_2(n)}}{G^{h(n)}} = E_{h_1,h_2},
\end{equation}
where $G$ is a nonzero constant only depending on $A$ and $E_{h_1,h_2}$ is a constant depending 
on $A$ and the sequences $h_1$ and $h_2$.
This  generalizes the results of  \cite{ERS}, where the case of $h_1(n)=0$ and 
very particular fractal sequences $h(n)=h_2(n)$ (therein referred to as {\em distinguished sequences}) was considered.

The  notion of fractal sequences is general enough to allow a ``complete'' understanding 
of the determinant asymptotics in the following sense. Namely, as another main result 
(referred to as the {\em uniform version of the limit theorem})
we will establish that 
\begin{equation}\label{Sz-limit-2}
\lim_{n_2-n_1\to \infty}
\left(\frac{ \det P_{n_1,n_2} A P_{n_1,n_2}}{G^{n_2-n_1}} - E[n_1,n_2]\right)=0
\end{equation}
under appropriate conditions on $A$. To make this a non-trivial statement, it is also established that 
 $E[n_1,n_2]=\Theta_{A,1}(\tau_{n_1})\Theta_{A,2}(\tau_{n_2})$
with explicitly defined elements $\tau_n$ of a compact Hausdorff space and 
{\em continuous} complex-valued functions $\Theta_{A,1}$ and $\Theta_{A,2}$.
The limit \eqref{Sz-limit-2}, being of the form $\lim\limits_{n_2-n_1\to \infty} \sigma_{n_1,n_2}=0$, is understood in the usual way, i.e., meaning that for each $\epsilon>0$ there exists $k\in\mathbb{N}$ such that $|\sigma_{n_1,n_2}|<\epsilon$
whenever $n_1,n_2\in\Z$ and $n_2-n_1>k$.

 \bigskip
Let us now introduce the notions of almost periodic sequences and band dominated operators with 
almost periodic diagonals, the latter being the class of operators $A$ for which the limits  
\eqref{Sz-limit} and \eqref{Sz-limit-2} will be established.

\medskip
{\bf Almost periodic sequences.}\
Denote by $AP(\Z)$ the set of {\em almost periodic sequences} consisting of all $a=\{a(k)\}_{k\in\Z} \in l^\infty(\Z)$ for which the set
    \[\{U_na:n\in \Z\}\]
is relatively compact in the norm topology of $l^\infty(\Z)$. Here
    \begin{equation}\label{Un on l00}
        U_n:a\in l^\infty(\Z) \mapsto b\in l^\infty(\Z), \qquad b(k):= a(k-n)
    \end{equation}
is the shift operator acting (isometrically) on $l^\infty(\Z)$. Despite the difference of the underlying spaces in (\ref{Un on l2}) and (\ref{Un on l00}), we will use the same symbol for brevity. There is an equivalent definition of $AP(\Z)$ as the closure in $l^\infty(\Z)$ of the set of all finite linear combinations of sequences $e_\xi\in l^\infty(\Z)$, where
\begin{equation}
e_\xi(k) = e^{2\pi i k\xi},\quad k\in \Z,\enskip \xi\in \R.
\end{equation}
Note that $e_\xi$ depends on $\xi\in \R$ only modulo $\Z$. Therefore, it is more appropriate
to think of $\xi$ as an element in  $\R/\Z$,  the additive group arising from $\R$ by identifying two numbers whose difference is an integer.

For each $a=\{a(k)\}_{k\in\Z}\in AP(\Z)$, its mean $M(a)$ is well-defined by the limit
    \begin{equation}\label{mean.Ma}
        M(a) = \lim_{n\to \infty} \frac{1}{n}\sum_{k=0}^{n-1}a(k).
    \end{equation}
The Fourier coefficients of a sequence $a\in AP(\Z)$ are defined by
    \begin{equation}
        a_\xi = M(ae_{-\xi}), \enskip \xi \in \R/\Z.
    \end{equation}
It is known that the set of all $\xi$ for which $a_\xi\neq 0$ is at most countable for each $a\in AP(\Z)$. 
This set is called the Fourier spectrum of $a$.
The theory of almost periodic sequences is similar to the theory of almost periodic functions on $\R$.
For details and basic information we refer to \cite{Corduneanu}.

\bigskip
{\bf Band-dominated operators with almost periodic diagonals.}\
We define the class  $\mathcal{OAP}$ of  \emph{operators with almost periodic diagonals} as the set of all bounded linear operators $A$ on $l^2(\Z)$ whose $n$-th diagonal $a^{(n)}$ belongs to $AP(\Z)$ for each $n\in \Z$. In other words,
\begin{equation}\label{def.D(A)}
   a^{(n)} = D_n(A):= D(AU_{-n})\in AP(\Z),
\end{equation}
where $D(A)\in l^\infty(\Z)$ stands for the main diagonal of a bounded linear operator $A$ on $l^2(\Z)$.
While a Laurent operator can be written as
    \[L(a) = \sum_{n\in \Z}a^{(n)}U_n\]
where $a^{(n)} \in \mathbb{C}$ are constants, an operator $A\in \mathcal{OAP}$ can be formally written as
    \begin{equation}\label{form of A}
        A = \sum_{n\in \Z} (a^{(n)}I)U_n
    \end{equation}
where $a^{(n)} = D_n(A)\in AP(\Z)$ is the sequence representing the $n$-th diagonal of $A$ and $aI$ stands for the multiplication operator generated by $a=\{a(k)\}_{k\in\Z}\in l^\infty(\Z)$, i.e., 
    \begin{equation}
        aI: \{x_k\}_{k\in \Z}\in l^2(\Z)\mapsto    \{a(k)x_k\}_{k\in \Z}\in l^2(\Z).
    \end{equation}
    
A subclass of $\mathcal{OAP}$ is the collection of \emph{band-dominated operators with almost periodic diagonals}, the notion of which is used rather loosely.  One can define it as the closure of the set of all band operators with almost periodic diagonals with respect to some appropriate norm. Band operators are operators of the form (\ref{form of A}) with the sum being finite. For more on the theory of band-dominated operators, see \cite{RaRo-07, RaRo-12, Ro08, RoSi-07}.

\bigskip
{\bf The strong Szeg\"o-Widom limit theorem.}\
The class $AP(\Z)$ includes periodic sequences as a special case. Indeed, the sequence $e_\xi$ is periodic if and only if 
$\xi \in \Q/\Z$.
An operator acting on $l^2(\Z)$ whose diagonals are periodic sequences with period $N$ can be identified
with a block Laurent operator, which is defined by \eqref{L(a)} with $a$ being an $N\times N$ matrix-valued
function whose Fourier coefficients $a_n$ are $N\times N$ matrices.

The classical strong Szeg\"o-Widom limit theorem \cite{Widom} describes the asymptotics of the determinants
of the finite sections $P_{n_1,n_2}L(a) P_{n_1,n_2}$ of block Laurent operators in the case
when $n_1=0$ and $n_2=nN$ as $n\to \infty$. 
Let us recall this result for generating functions in $B^{N\times N}$,  where  $B=W\cap F\ell^{2,2}_{1/2,1/2}$
is the Banach algebra of all $a\in L^\iy(\T)$ for which 
$$
\|a\|_{B}:= \sum_{n=-\infty}^\infty |a_n|+\left(
\sum_{n=-\infty}^\infty |n| \cdot |a_n|^2\right)^{1/2}<\infty.
$$
Further information and  different versions of this limit theorem for various classes of functions 
can be found in \cite{BS99, BS-ana, Eh-SzW, RoSi-07, SimonB} and in the references given there.

\begin{thm}[Szeg\"o-Widom limit theorem]\label{t1.1}
For $a\in B^{N\times N}$ assume that $\det a(t)\neq0$ for all $t\in\T$ and that $\det a(t)$ has winding number zero.
Then 
\begin{equation}\label{SW-lim}
\lim_{n\to \infty} \frac{\det P_{0,nN} L(a) P_{0,nN}}{G[a]^n}=E[a],
\end{equation}
where $G[a]=\exp\left(\frac{1}{2\pi}\int\limits_0^{2\pi} \log\det a(e^{ix})\, dx\right)$ and $E[a]=\det T(a) T(a^{-1})$.
\end{thm}
Therein $T(a)=(a_{j-k})_{j,k=0}^\infty$ is the block Toeplitz operator thought of as acting on $l^2(\Zp)$.
The constant $E[a]$ is given in terms of a well-defined operator determinant.
In the scalar case ($N=1$) the constant $E[a]$ admits a more explicit expression and is always nonzero.
In general ($N\ge2$),  the constant $E[a]$ can happen to be zero and explicit expressions are known 
only in very special cases (e.g., if $a$ or $a^{-1}$ are trigonometric matrix polynomials).

\medskip
The classical Szeg\"o-Widom limit theorem can be used to determine the asymptotics of 
$\det P_{n_1,n_2}L(a)P_{n_1,n_2}$  when the block size fits the size of the finite sections, i.e., when
$N$ divides $n_2-n_1$.
Indeed, for fixed $k\in\{0,\dots,N-1\}$ the identity
$$
\det P_{k,k+nN} L(a) P_{k,k+nN} =\det P_{0,nN} L(a^{[k]}) P_{0,nN} 
$$
holds, where $a^{[k]}$ are slightly modified symbols determined by  $U_{-k} L(a) U_k = L(a^{[k]})$.
These symbols yield the same constant $G$ but in general different values for $E$. 
This observation corroborates the statement made above that one cannot expect an asymptotics
of the kind \eqref{no-limit}, and gives an indication of the notion of fractal sequences in this case.

In Subsection \ref{sec:7.3}, as a consequence to our main results, we will obtain a limit theorem for the determinants 
$\det P_{n_1,n_2} L(a) P_{n_1,n_2}$ where the condition that $N$ divides $n_2-n_1$ is removed.

\bigskip
{\bf The almost Mathieu operators}.\  
The other most prominent example of operators to which the main results can be applied (under certain conditions) is the almost Mathieu operator defined on $l^2(\Z)$, 
\begin{equation}\label{Mathieu}
M_a =  U_1 +aI+U_{-1},
\end{equation}
where $a\in AP(\Z)$ is given by $a(n)=\beta\cos 2\pi(\xi n +\delta)$, and $\beta$, $\xi$, and $\delta$ are certain (real) constants. The interesting case is when $\xi$ is irrational. Then the spectrum of the almost Mathieu operator
is a Cantor-like set. This was an open conjecture attracting quite some attention before it was finally proved 
by Avila and Jitomirskaya \cite{AvilaJit}. For more on the history and on  general information about 
the almost Mathieu operators see \cite{AvilaJit,Boca} and the references therein.

One can ask the question for which values of $\xi$, $\beta$, $\delta$, and $\lambda$,
it is possible to apply our results to $A=M_a-\lambda I$ and obtain a corresponding asymptotics
\eqref{Sz-limit} and \eqref{Sz-limit-2}. One restriction is that $\xi$ is not a Liouville number, another one
is that $A$ is invertible on $l^2(\Z)$. These restrictions can probably not be relaxed very much.
If a conjecture raised in  Subsection \ref{sec:7.4} turns out to be true, then these are  the only restrictions
(assuming $M_a$ to be self-adjoint).

Regardless of whether the conjecture is true, we can still say that for non-Liouville numbers $\xi$,
 our results can be applied to $A=M_a-\lambda I$ whenever $|\lambda|$
is sufficiently large. For details see again Subsection \ref{sec:7.4}.

\bigskip
{\bf The notion of fractal sequences.}\
In Sections \ref{sec:2} and \ref{sec:3} we will introduce the notion of a fractal sequence. It is an integer sequence
$h:\Zp\to\Z$, i.e., $h=\{h(n)\}_{n\ge0}$, characterized by a property that guarantees the existence of certain limits
in various settings. More precisely, we will encounter 
\begin{enumerate}
\item[(i)]
fractal sequences for an additive subgroup $\Xi$ of $\R/\Z$,
\item[(ii)]
fractal sequences for a Banach subalgebra $\cA$ of $l^\iy(\Z)$,
\item[(iii)]
fractal sequences for a Banach subalgebra $\cR$ of $\cL(l^2(\Z))$.
\end{enumerate}
The name ``fractal'' is chosen because (in the appropriate setting, see Lemma \ref{frac subsequence}) each integer sequence has subsequence which is fractal. 
Furthermore, there seems to be a connection with the work of Roch and Silbermann 
on fractal algebras of approximation sequences  (see \cite{RoSi-96, Ro18} and the references therein). Although their notion of fractality is used in the setting of $C^*$-algebras, it can be given for Banach algebras as well. 
It is not too hard to see that if $h_1$ and $h_2$ are fractal sequences (in our sense), then the Banach algebra
$\mathcal{S}_{h_1,h_2}$ to be defined in Theorem \ref{thm:4.4} is fractal in the sense of Roch/Silbermann.
Note that this Banach algebra plays the crucial role in the proof of our main results in Section \ref{sec:5}.

\bigskip
{\bf On the assumptions encountered in the limit theorems.}\ 
In our main results, the fractal and the uniform version of the limit theorem
(Theorems  \ref{MainResult-1} and \ref{MainResult-2}), several assumptions have to be imposed.
One may wonder about the form of these assumptions and whether they are really necessary.
We want to argue here that they are at least close to be necessary. 
We will encounter three kinds of assumptions on the ``symbol'' $A$ that occurs in the limits
\eqref{Sz-limit} and \eqref{Sz-limit-2}, which could be referred to as
\begin{itemize}
\item[(i)]
a smoothness assumption,
\item[(ii)]
a regularity assumption,
\item[(iii)]
a diophantine assumption.
\end{itemize}
The first two already occur naturally in the Szeg\"o-Widom limit theorem, where $A=L(a)$.
Therein first condition comes down to assume the function $a$ to be sufficiently smooth. 
The second one is the requirement that $\det a(t)$ does not vanish and has winding number zero. Equivalently, this means that $\det a$ has a continuous logarithm or that the matrix function $a$ is the finite product of
exponentials. (To be precise, the equivalence holds for various classes of smooth functions.  It breaks down
for the rather exotic classes of Krein algebras. For details see Proposition 6.4 and the concluding remarks in \cite{Eh-SzW}.)  

In our limit theorems the first assumption corresponds to assuming a sufficiently fast decay of the diagonals of $A$
(therefore the restriction  to band-dominated operators) and in addition a sufficiently fast decay of the Fourier coefficients  of the almost periodic sequences that occur as the diagonals of $A$. More specifically, we will consider weighted Wiener-type  algebras of such operators, which involve an ``admissible''  weight $\beta$ on an additive
subgroup $\Xi\subseteq \R/\Z$ containing the Fourier spectra of the diagonals of $A$.

The regularity assumption in our limit theorems comes down to assuming that $A$ is a finite product of exponentials, i.e., 
$A=e^{A_1}\cdots e^{A_r}$, where $A_1,\dots,A_r$ have to belong to the afore-mentioned Wiener-type algebras.
This is stronger than invertibility, but seems natural in view of what is assumed in the classical Szeg\"o-Widom limit theorem. For another interpretation of this condition and its relation to inverse closedness, see
Section \ref{sec:7.4}.

Finally, the diophantine assumption is peculiar to the almost periodic case. It occurs in the form that the weight 
$\beta$ has to be ``compatible''.  The {\em compatibility condition} involves the underlying group $\Xi$ of the Fourier spectra,
and it is conceivable that not every $\Xi$ has an admissible and compatible weight. However, if the elements of  $\Xi$
have ``nice''  diophantine properties then the existence of such weights is guaranteed  by non-trivial results from
diophantine approximation (see the beginning of Section \ref{sec:7} and Subsection \ref{sec:7.2} in particular for details).
As already demonstrated in \cite{ERS} and as we will indicate next, without the compatibility condition the asymptotics of the determinants may not take the desired form, the reason being that the trace computations break down.

\bigskip
{\bf Method of proof: reduction of determinants to traces, and the compatibility condition.}\ 
Let us make one important point pertaining to the method of proof and
explaining how the compatibility condition arises from the asymptotics of the traces.
First of all, if one wants to study the asymptotics of the determinants
 $\det P_{h_1(n),h_2(n)} A P_{h_1(n),h_2(n)}$ for band-dominated operators $A$ one has to be able to describe the  asymptotics of the traces of the finite sections 
of such operators in the first place. Indeed, this can be seen by considering the special case
of a diagonal operator $A= e^{a}I$ with $a\in AP(\Z)$ and notice that 
$$
\det  \left(P_{h_1(n),h_2(n)} (e^a I) P_{h_1(n),h_2(n)}\right)
=\exp\left( \trace  P_{h_1(n),h_2(n)} (a I) P_{h_1(n),h_2(n)}\right).
$$
Conversely, our method of establishing the determinant asymptotics consists in assuming that 
$A=e^{A_1}\cdots e^{A_r}$, where $A_k$ are operators taken from certain subclasses of $\mathcal{OAP}$, and then reducing asymptotics of $\det P_{h_1(n),h_2(n)} A P_{h_1(n),h_2(n)}$ to the asymptotics of the traces
$$
\trace\left(P_{h_1(n),h_2(n)}(A_1+\dots+A_r)P_{h_1(n),h_2(n)}\right).
$$
This idea is realized in the first of our main results, the {\em abstract version of the limit theorem} (Theorem
\ref{MainResult-0}). Notice that the above trace equals  
\begin{equation}\label{trace.sum}
  \trace (P_{h_1(n),h_2(n)} (aI) P_{h_1(n),h_2(n)})=\sum_{k=h_1(n)}^{h_2(n)-1}a(k),
\end{equation}
where  $a=D(A_1+\dots+A_r)\in AP(\Z)$. Thus, to summarize, we are led to the problem of
describing the asymptotics of traces \eqref{trace.sum} for $a=\{a(k)\}_{k\in\Z}\in AP(\Z)$.

Since we assume that $h(n):=h_2(n) - h_1(n)\to \infty$ as $n\to\infty$, the theory of almost periodic sequences implies that
 \eqref{trace.sum} equals 
$$
h(n)\cdot M(a) + o(h(n)), \qquad n\to\iy, 
$$
where $M(a)$ is the mean of $a\in AP(\Z)$ defined in \eqref{mean.Ma}.  However, what we would like to have 
(in view of the desired asymptotics \eqref{Sz-limit})
is an asymptotics of the kind
\begin{equation}\label{tr-asymp}
h(n)\cdot M(a) + C_{h_1,h_2} + o(1), \qquad n\to\iy,
\end{equation}
with some constant $C_{h_1,h_2}$ possibly depending on the underlying sequences $h_1$ and $h_2$.
In order to see when we can expect such a behavior,  consider the case where $a\in AP(\Z)$ is a finite sum of the form 
\begin{equation}\label{aps-1}
a=\sum_{\xi} a_\xi e_\xi,\qquad a_\xi\in\C. 
\end{equation}
Then we have (for details see the proof of Theorem \ref{thm trace eva} below)
$$
    \sum_{k=h_1(n)}^{h_2(n)-1}a(k)  
   = h(n)\cdot M(a) + \sum_{\xi\neq0}a_\xi\frac{e^{2\pi i h_1(n)\xi}-e^{2\pi i h_2(n)\xi}}{1-e^{2\pi i\xi}}.
$$
In order for the last term to converge as $n\to\iy$ one should assume that the sequences $h_1$ and $h_2$ are such that the limits
$$
\lim_{n\to\infty} e^{2\pi i h_j(n)\xi}=:\tau_j(\xi),\qquad j=1,2,
$$
exist for each $\xi$ over which the summation is taken. This condition will later be conceptionalized by saying that
$h_1$ and $h_2$ are {\em fractal sequences}. If this holds then the trace equals
\begin{equation}\label{f1.20}
h(n)\cdot M(a) + \sum_{\xi\neq0}a_\xi\frac{\tau_1(\xi)-\tau_2(\xi)}{1-e^{2\pi i\xi}}+o(1), \qquad n\to\iy,
\end{equation}
where the second term is a well-defined  constant $C_{h_1,h_2}$ depending on $h_1$ and $h_2$.

Things get more complicated if we consider $a\in AP(\Z)$ for which the sum \eqref{aps-1} involves infinitely many terms. Then we encounter a ``small denominator problem'' in the expression for the traces and in its expected asymptotics \eqref{f1.20}. 
In some cases, when the Fourier spectrum of $a$ is ``non-Liouville'' and the decay of the Fourier coefficients $a_\xi$ is 
sufficiently fast we can still guarantee the asymptotics \eqref{f1.20}. A corresponding result is established
in Theorem \ref{thm trace eva}, which involves the {\em compatibility condition}.
The relationship with diophantine properties will be discussed in Section \ref{sec:7} 
(see Subsection \ref{sec:7.1} in particular). 

On the other hand, there exists $a\in AP(\Z)$ for which an asymptotics of the form
\eqref{tr-asymp} does not hold, even if we assume $h_1$ and $h_2$ to be fractal sequences.
This was shown by explicit (though complicated) examples in Section 6 of  \cite{ERS}, 
and will be briefly mentioned at the beginning of Section \ref{sec:7}.
Therefore, one cannot completely  eliminate the diophantine assumption in the
formulation of our limit theorems.

\bigskip
{\bf Outline of the paper.}
The paper is organized as follows. Section 2 
deals with Banach algebras of almost periodic sequences. The notions of fractality and of admissible and compatible
weights on an additive subgroup $\Xi$ of $\R/\Z$ are introduced as well.

In Section 3 we consider Banach algebras of operators on $l^2(\Z)$ which are characterized by properties called
suitablity, shift-invariance and rigidity. Based on Section 2, concrete examples of such Banach algebras are described.

Sections 4 and 5 are devoted to the abstract version of the limit theorem and its proof. Note that the abstract version
applies to operators $A$ taken from any suitable, shift-invariant, rigid, and unital Banach algebra of operators on 
$l^2(\Z)$. The proof is based on a Banach algebra method introduced by one of the authors in \cite{Eh-SzW} for the 
classical  Szeg\"{o}-Widom limit theorem. This approach was also employed in \cite{ERS}. 

At the end of Section 5 we establish the fractal version of the limit theorem using the previous results.
In Section 6 we obtain the uniform version, and we show that certain quantities appearing in the limit are continuous functions.

Section 7 discusses a variety of issues that naturally arise when trying to apply the limit theorems to concrete operators. 
In particular we will discuss the case of finitely generated groups and explain the relationship between the compatibility condition and diophantine properties. We will also establish a generalization of the block Szeg\"o-Widom limit theorem
dealing with ``non-standard'' finite sections of Laurent operators. Finally a conjecture concerning an inverse closedness problem is raised, and the applicability of our results to almost Mathieu operators is discussed. 

\section{\Large\bf Banach algebras of almost periodic sequences}
\label{sec:2}

A Banach subalgebra $\mathcal{A}$ of $l^\infty(\Z)$ is called \emph{shift-invariant} if\
 for each $a\in \mathcal{A}$ and $n\in \Z$ we have $U_na\in \mathcal{A}$ and
    \[\|U_na\|_\mathcal{A} = \|a\|_\mathcal{A}.\]

Let $\mathcal{A}$ be a shift-invariant Banach subalgebra of $l^\infty(\Z)$. A sequence $h:\Zp\to \Z$ is called \emph{fractal for $\mathcal{A}$} if for each $a\in\cA$ there exists an element $Ua\in \cA$ such that 
\begin{equation}
        \lim_{n\to \infty}\|U_{-h(n)}a - U a\|_\mathcal{A} = 0.\label{eqn for fractal}
\end{equation}
It is easy to see that then the map $U:\cA\to\cA$ is an isometric Banach algebra homomorphism. 
Indeed, note that $U$ is multiplicative because each $U_n$ acts multiplicatively on $\cA$.
If $\cA$ contains the unit element of $l^\infty(\Z)$, then $U$ is unital.
As the sequence $h$ determines uniquely the map $U$ by \eqref{eqn for fractal},
 we will also say that $h$ is \emph{fractal for $\cA$ with associated $U$}.

\medskip
In what follows we are going to describe shift-invariant Banach subalgebras $\cA$ of $AP(\Z)\subseteq l^\infty(\Z)$ which 
are characterized by an additive subgroup of $\R/\Z$. For this purpose we extend our conceptual framework
with the following definitions. 

\medskip
For an additive subgroup $\Xi$ of $\R/\Z$ denote by 
$\HomXT$ the set of all group homomorphisms $\tau:\Xi\to\T$. 
Recall that  $\T=\{t\in\C\,:\,|t|=1\}$ is the (multiplicative) circle group.
Then  $\HomXT$ is a compact abelian group with group multiplication defined by
$(\tau_1 \tau_2)(\xi)=\tau_1(\xi)\tau_2(\xi)$. The topology of $\HomXT$ arises from the local bases at $\tau$ given
by the collection of all neighborhoods
\begin{equation}\label{neigborhood}
U_{\xi_1,\dots,\xi_N;\varepsilon}[\tau]=\Big\{\, \tau'\in\HomXT \,:\, |\tau'(\xi_k)-\tau(\xi_k)|<\varepsilon \mbox{ for all } 1\le k\le N\,\Big\}
\end{equation}
with $\varepsilon>0$, $N\in\mathbb{N}$, $\xi_1,\dots,\xi_N\in\Xi$. Note that if we consider $\Xi$ with discrete topology,
then $\HomXT$ is the dual group of $\Xi$.

Given an additive subgroup $\Xi$ of $\R/\Z$, we say that a sequence $h:\Zp\to \Z$ is \emph{fractal for $\Xi$} if for each 
$\xi\in\Xi$ the limit
\begin{equation}\label{fractal-tau}
		\tau(\xi) := \lim\limits_{n\to\infty} e^{2\pi i h(n)\xi}
\end{equation}
exists. Obviously, in this case, $\tau\in \HomXT$. Therefore we will say that the sequence $h$ is \emph{fractal for $\Xi$ with associated $\tau\in \HomXT$}.

We can give another interpretation of \eqref{fractal-tau}. To each $n\in\Z$ there exists a naturally associated
$\tau_n\in \HomXT$ defined by $\tau_n(\xi)=e^{2\pi i  n \xi}$. Using this notation, 
$h$ being fractal for $\Xi$ is equivalent  to
saying that $\tau_{h(n)}\to \tau$ in the topology of $\HomXT$ as $n\to\iy$.

\begin{prop}\label{prop fractal A} Let $\Xi$ be an additive subgroup of $\R/\Z$, and let $\mathcal{A}$ be a shift-invariant Banach subalgebra of $AP(\Z)$ such that the linear span of
    \[\{e_\xi:\xi \in \Xi\}\]
is contained and dense in $\mathcal{A}$.  Let $h:\Zp\to\Z$.
\begin{enumerate}[label=(\alph*)]
\item
The sequence $h$ is fractal for $\cA$ if and only if it is fractal for $\Xi$.
\item
If this is true, and  if $h$ is fractal for $\Xi$ with associated $\tau\in \HomXT$, then $h$ is fractal for $\cA$ with the
associated $U$ given by 
\begin{equation}\label{Utau on A}
U: \mathcal{A}\to \mathcal{A}, \quad
\sum_\xi a_\xi e_\xi \mapsto \sum_\xi a_\xi \tau(\xi)e_\xi.
\end{equation}
Therein, the operation is defined for finite linear combinations and extends by continuity to all of $\cA$.
\end{enumerate}
\end{prop}
\begin{proof}
Let us first show that if $h$ is fractal for $\cA$ with associated $U$, then it is fractal for $\Xi$.
Indeed, for each $\xi\in\Xi$, 
    \[\|U_{-h(n)}e_\xi - Ue_\xi\|_\mathcal{A} = \|e^{2\pi i h(n)\xi}e_\xi - Ue_\xi\|_\mathcal{A} \to 0\]
as $n\to\infty$. This implies that $e^{2\pi i h(n)\xi} e_\xi$ is a Cauchy sequence in $\cA$, and therefore
so is the scalar sequence $e^{2\pi i h(n)\xi}$. Hence the limit \eqref{fractal-tau} exists.
In fact, we have $Ue_\xi=\tau(\xi) e_\xi$.

Now let us assume that $h$ is fractal for $\Xi$ with associated $\tau\in \HomXT$. We are going to show that 
$h$ is fractal for $\cA$ with associated $U$ given above. First consider those $a\in \cA$ which can be written as
a finite linear combination $a=\sum_{\xi} a_\xi e_\xi$. Since $U_{-h(n)}e_\xi=e^{2\pi i h(n) \xi} e_\xi$ it follows that 
$$
U_{-h(n)} a = \sum_{\xi} a_\xi e^{2\pi i h(n)} e_\xi \to  \sum_{\xi} a_\xi \tau(\xi) e_\xi =: Ua,
$$
as $n\to \infty$, where the convergence is in the norm of $\cA$. Considering $U_{-h(n)}$ as a bounded linear operator
on the Banach space $\cA$ we thus have strong convergence on the dense subset of $\cA$. Because of shift-invariance, i.e., 
$\|U_{-h(n)}a\|_{\cA} = \|a\|_{\cA}$, the operator norm of $U_{-h(n)}$ equals one. A standard approximation argument 
implies that we have strong convergence of $U_{-h(n)}$ on all of $\cA$, and that the operator $U$ (already defined
on a dense subset) extends by continuity to all of $\cA$. In fact, we have $\|Ua\|_{\cA}=\|a\|_{\cA}$.
\end{proof}

Let us now proceed with describing a concrete class of shift-invariant Banach algebras  $\cA$
 to which the previous proposition can be applied. They arise from any additive subgroup $\Xi$ of $\R/\Z$ and
 can be considered as a weighted Wiener-type algebra of almost periodic sequences. 
 We call a mapping $\beta: \Xi\to \R^+$ an \emph{admissible weight on $\Xi$} if
\begin{equation}\label{weight.adm}
1\le \beta(\xi_1+\xi_2)\le \beta(\xi_1)\beta(\xi_2)
\end{equation}
for each $\xi_1, \xi_2\in \Xi$. For such $\Xi$ and $\beta$ let $APW(\Z, \Xi, \beta)$ be the set of all sequences $a\in l^\infty(\Z)$ of the form
\begin{equation}  \label{APWsum}
     a = \sum_{\xi \in \Xi} a_\xi e_\xi
\end{equation}
for which
        \begin{equation}\label{APWnorm}
            \|a\|_{\Xi, \beta} := \sum_{\xi\in\Xi}\beta(\xi)|a_\xi|<\infty.
        \end{equation}
As usual, in the case when $\Xi$ is uncountable it is agreed that at most countably many of the $a_\xi$'s 
are nonzero, and only over those the sum is taken.

The following result is almost obvious and was proved in \cite[Thm.~2.6]{ERS}.
It implies that Proposition \ref{prop fractal A}  can be applied to $\cA=APW(\Z, \Xi, \beta)$. 
In addition, it is easy to see that the map $U$ is given by formula \eqref{Utau on A} for all $a\in APW(\Z, \Xi, \beta)$.

\begin{prop}\label{p2.2}
Let $\beta$ be an admissible weight on an additive subgroup $\Xi$ of $\R/\Z$. Then $APW(\Z, \Xi, \beta)$ is a shift-invariant, continuously embedded Banach subalgebra of $AP(\Z)$, and the linear span of $\{e_\xi:\xi\in \Xi\}$ is a dense subset.
\end{prop}

We need one more property about weights. A weight $\beta$ is said to be \emph{compatible on} $\Xi$ if
\begin{equation}\label{compat.cond}
C_\beta:= \inf_{\xi\in \Xi, \xi \neq 0} \beta(\xi)\cdot\|\xi\|_{\R/\Z} >0.
\end{equation}
Therein, we use the natural metric on $\R/\Z$ given by
\begin{equation}
\|\xi\|_{\R/\Z} = \inf \{|\xi - n|:n\in \Z\}.
\end{equation}
The next theorem gives us information about the asymptotics of the trace of the finite sections
$P_{h_1(n),h_2(n)}(aI)P_{h_1(n),h_2(n)}$ under certain conditions (see also \eqref{trace.sum}).

\begin{thm}\label{thm trace eva}
Let $\beta$ be an admissible and compatible weight on an additive subgroup $\Xi$ of $\R/\Z$. If $h_1,h_2:\Zp\to\Z$ 
are two fractal sequences for $\Xi$ with associated $\tau_{1},\tau_{2}\in \mathrm{Hom}(\Xi, \mathbb{T})$, respectively, then
for each $a\in APW(\Z,\Xi,\beta)$ we have
        \begin{equation}\label{trace-conv}
            \sum_{k=h_1(n)}^{h_2(n)-1} a(k) = \left(h_2(n) - h_1(n)\right)\cdot M(a) + F_a(\tau_{1}) - F_a(\tau_{2})+ o(1)
        \end{equation}
as $n\to \infty$, where 
	\begin{equation}\label{C-tau}
		F_a(\tau):= \sum\limits_{\xi\in \Xi, \xi\neq 0}a_\xi \, \frac{\tau(\xi)}{1-e^{2\pi i \xi}}
	\end{equation}
is a well defined constant depending on $a$ and $\tau\in \mathrm{Hom}(\Xi, \mathbb{T})$.
\end{thm}
\begin{proof}
First note that $F_a(\tau)$ is well-defined for each $\tau\in \mathrm{Hom}(\Xi, \T)$ with $a$ given by 
\eqref{APWsum}. Indeed, the corresponding series 
\eqref{C-tau} is absolutely convergent since $\beta$ is compatible, since \eqref{APWnorm} is finite, and since
\begin{equation}\label{estimate-beta}
\frac{1}{|1-e^{2\pi i \xi}|} =\frac{1}{2|\sin(\pi \xi)|}\le \frac{1}{4\, \|\xi\|_{\R/\Z}} \le \frac{\beta(\xi)}{4C_\beta},
\qquad \xi\in\Xi,\xi\neq0.
\end{equation}
Recall that we identify $\xi\in\R/\Z$ with any real number representing it.

Now let $a\in APW(\Z,\Xi,\beta)$ be of the form \eqref{APWsum}. Since
$$
            \sum_{k=h_1(n)}^{h_2(n)-1} a(k)  = \sum_{\xi\in \Xi}a_\xi \left(\sum_{k=h_1(n)}^{h_2(n)-1}e_\xi(k)\right),
$$
$M(a) = a_0$ and $e_\xi(k)=e^{2\pi i k \xi}$, we obtain that 
\begin{align*}
            \sum_{k=h_1(n)}^{h_2(n)-1} a(k) - (h_2(n) - h_1(n))\cdot M(a)  &= \sum_{\xi\in \Xi, \xi \neq 0}a_\xi \left(\sum_{k=h_1(n)}^{h_2(n)-1}e_\xi(k)\right)\\
            &= 
            \sum_{\xi\in \Xi, \xi \neq 0}a_\xi   \,
             \frac{e^{2\pi i h_1(n)\xi}-e^{2\pi i h_2(n)\xi}}{1-e^{2\pi i\xi}}.
\end{align*}
Since $h_1$ and $h_2$ are fractal for $\Xi$ with associated $\tau_1$ and $\tau_2$ we have that for each fixed $\xi\in \Xi$,
$\xi\neq0$,
        \[\frac{e^{2\pi i h_1(n)\xi}-e^{2\pi i h_2(n)\xi}}{1-e^{2\pi i\xi}} 
        \;\;\to\;\; 
        \frac{\tau_1(\xi)-\tau_2(\xi)}{1-e^{2\pi i\xi}}\]
as $n\to \infty$ by (\ref{fractal-tau}). 
Now use the estimate \eqref{estimate-beta} together with a dominated convergence argument and the finiteness of \eqref{APWnorm} to see that 
$$
 \sum_{\xi\in \Xi, \xi \neq 0}a_\xi   \,
             \frac{e^{2\pi i h_1(n)\xi}-e^{2\pi i h_2(n)\xi}}{1-e^{2\pi i\xi}}
\to              
\sum_{\xi\in \Xi, \xi \neq 0}a_\xi   \,
             \frac{\tau_1(\xi)-\tau_2(\xi)}{1-e^{2\pi i\xi}}=F_a(\tau_1)-F_a(\tau_2)
$$
as $n\to\infty$. This concludes the proof of \eqref{trace-conv}.
\end{proof}

Note that $F_a$ can be considered as a function in $\tau\in\HomXT$.
In fact, we will show later in Proposition \ref{C is cont} that $F_a$ is continuous on $\HomXT$ for each fixed $a$.

\section{\Large\bf Banach algebras of operators on $\boldsymbol{l^2(\Z)}$}
\label{sec:3}

In this section we are going to characterize certain classes of Banach algebras of operators on $l^2(\Z)$ for which we will prove an abstract  version of a Szeg\"{o}-Widom type limit theorem in Section \ref{sec:5}.
The proof of this limit theorem is based on a ``Banach algebra approach'', which was introduced in \cite{Eh-SzW} and used also in \cite{ERS}. The main goal of the approach is to reduce the asymptotics of determinants to the asymptotics of traces of certain operators.

The operators $A$ that belong to these classes of Banach algebras can be considered as ``symbols'' for the corresponding finite sections $A_{n_1,n_2}=P_{n_1,n_2} A P_{n_1,n_2}$. Moreover, these operators also serve as ``symbols''
for corresponding compression operators on $l^2(\Zp)$, quite analogous to the classical case of symbols of Toeplitz operators.

The main property that such Banach algebras have to possess is \emph{suitability} (in the sense of \cite{Eh-SzW} or  \cite{ERS}). In addition, the notions of \emph{shift-invariance} and  \emph{rigidity} are needed, 
and the notion of  {\em fractal sequences} will occur again.

Let $\mathcal{L}(H)$ denote the Banach algebra of all bounded linear operators on a Hilbert space $H$.
Furthermore, let $P, J$ stand for the following operators on $l^2(\Z)$,
    \begin{align*}
        P&:\enskip (x_n)_{n\in \Z}\mapsto (y_n)_{n\in \Z}\enskip \mathrm{with}\enskip y_n =\begin{cases}
        x_n& \mathrm{if}\enskip n\ge 0\\
        0  &   \mathrm{if}\enskip n<0,
        \end{cases}\\
        J&:\enskip (x_n)_{n\in \Z} \mapsto (x_{-n-1})_{n\in \Z}.
    \end{align*}
For each operator $A\in \mathcal{L}(l^2(\Z))$, define
    \[T(A):=PAP, \quad H(A):=PAJP, \quad \widetilde{A} := JAJ.\]
Identifying the image of $P$ with $l^2(\Zp)$, we will consider $T(A)$ and $H(A)$ as operators acting on $l^2(\Zp)$. 
The notation above is inspired by the classical notation for Toeplitz and Hankel operators.
In fact, the following simple, but important identities hold for any $A,B\in \mathcal{L}(l^2(\Z))$,
    \begin{align}\label{T-id}
        T(AB) &= T(A)T(B)+H(A)H(\widetilde{B}),\\
        H(AB) &= T(A)H(B)+H(A)T(\widetilde{B}),
    \end{align}
generalizing the  classical identities for Toeplitz and Hankel operators.

\bigskip
{\bf Rigidity and suitability.}\ 
A set $\mathcal{R}$ of bounded linear operators on $l^2(\Z)$ is called \emph{rigid} if for each $A\in \mathcal{R}$ the following statement holds:
    \[\text{If}\enskip T(A) \enskip \text{or} \enskip T(\widetilde{A}) \text{ is compact, then }A =0.\]
The notion is modelled after the corresponding property for Toeplitz operators.
It was proved in  \cite[Thm.~3.1]{ERS} that the class $\mathcal{OAP}$ is rigid.

\medskip
A Banach algebra $\mathcal{R}$ of bounded linear operators on $l^2(\Z)$ will be called \emph{suitable} if the following conditions hold:\hfill
\begin{enumerate}[label=(\alph*)]
\item $\mathcal{R}$ is continuously embedded into $\mathcal{L}(l^2(\Z))$ and
    \[\|A\|_{\mathcal{L}(l^2(\Z))} \le \|A\|_\mathcal{R}\text{  for all }A\in \mathcal{R}.\]
\item
For all $A,B\in \mathcal{R}$, the operators
$H(A)H(\widetilde{B})$ and $H(\widetilde{A})H(B)$ are trace class, and there exists $M>0$ such that 
    	\begin{equation}\label{suitable}
		\max\{\|H(A)H(\widetilde{B})\|_{\mathcal{C}_1}, \|H(\widetilde{A})H(B)\|_{\mathcal{C}_1}\} \le M\|A\|_\mathcal{R}\|B\|_\mathcal{R}\text{   for all }A,B\in \mathcal{R}.
	\end{equation}
\end{enumerate}
Therein, for $1\le p<\infty$, let $\mathcal{C}_p(H)$ stand for the Schatten-von Neumann class of operators on a Hilbert space $H$, i.e., the set of all (compact) operators $A\in \mathcal{L}(H)$ for which
    \[ \|A\|_{\mathcal{C}_p}:= \left( \sum_{n\ge 0}s_n(A)^p \right) ^{1/p}<\infty, \]
where $s_n(A)$ refers to the $n$-th singular value of $A$. The operators belonging to $\mathcal{C}_1$ are 
called trace class operators. We refer to \cite{gohberg2013classes} for more information about these concepts.

\medskip

For sake of illustration, let us present an example of a class of suitable Banach algebras
(see also \cite[Example 3.2 and Prop.~3.3]{ERS}). 

\begin{example}\label{ex3.1}
For $p,q\ge 1$, define
\[
\mathcal{R}_{p,q} :=
            \left\{\, 
            A\in \mathcal{L}(l^2(\Z)): H(A)\in \mathcal{C}_p(l^2(\Zp)), \enskip H(\widetilde{A})\in \mathcal{C}_q(l^2(\Zp)) 
            \,\right\}
\]
along with a norm
    \[\|A\|_{\mathcal{R}_{p,q}}:= \|A\|_{\mathcal{L}(l^2(\Z))} + \|H(A)\|_{\mathcal{C}_p} + \|H(\widetilde{A})\|_{\mathcal{C}_q}.\]
With the above norm, $\mathcal{R}_{p,q}$ is a Banach algebra, and it is suitable if, in addition, $1/p+1/q = 1$.
\end{example}

\bigskip
{\bf Shift-invariance and fractal sequences.}\
Finally, let us define the notions of shift-invariance and fractal sequences for Banach algebras of operators on $l^2(\Z)$. A Banach algebra $\mathcal{R}$ of bounded linear operators on $l^2(\Z)$ is said to be \emph{shift-invariant} if
    \begin{equation}
        U_{-n}AU_n\in \mathcal{R}\enskip\text{   and   } \enskip\|U_{-n}AU_n\|_\mathcal{R} = \|A\|_\mathcal{R}
    \end{equation}
for each $A\in \mathcal{R}$ and $n\in \Z$. Occasionally, we will also use the notation 
\begin{equation}\label{Un}
        \mathcal{U}^{n}:A\in \mathcal{R}\mapsto U_{-n}AU_{n} \in \mathcal{R},
\end{equation}
noting that $\cU^n$ is an isometric Banach algebra isomorphism on $\cR$.

Assuming $\mathcal{R}$ to be shift-invariant, we call a sequence $h:\Zp\to \Z$ to be \emph{fractal} for $\mathcal{R}$ if for each $A\in\cR$ there exists an element $\mathcal{U}A\in \cR$ such that 
    \begin{equation}\label{eqn for fractal R}
        \lim_{n\to \infty} \|U_{-h(n)}AU_{h(n)} - \mathcal{U}A\|_\mathcal{R} = 0.
    \end{equation}
It is easy to see that in this case $\mathcal{U}:\cR\to\cR$ is an isometric Banach algebras homomorphism, which is unital
if $\cR$ is unital. We will also say that the sequence $h:\Zp\to\Z$ is {\em fractal with associated Banach algebra homomorphism} 
$\mathcal{U}$ on $\cR$.

\bigskip
{\bf The Banach algebras $\boldsymbol{\mathcal{W}_{\alpha_1, \alpha_2}(\mathcal{A})}$.} 
\ 
Let $\mathcal{A}$ be a shift-invariant and continuously embedded Banach subalgebra of $l^\infty(\Z)$. 
We are going to introduce corresponding Banach algebras $\mathcal{W}_{\alpha_1, \alpha_2}(\mathcal{A})$ of operators on $l^2(\Z)$ as follows. For given $\alpha_1, \alpha_2\ge 0$, define the weight function $\alpha$ on $\Z$ by
    \begin{equation}\label{alpha.k}
    \alpha(k) = \begin{cases}
                    (1+k)^{\alpha_1} &\mathrm{ if }\enskip k\ge 0, \\[.5ex]
                    (1+|k|)^{\alpha_2} &\mathrm{ if }\enskip k<0.
                  \end{cases}
            \end{equation}
Now let $\mathcal{W}_{\alpha_1, \alpha_2}(\mathcal{A})$ refer to the set of all operators $A\in \mathcal{L}(l^2(\Z))$ for which $D_k(A)\in \mathcal{A}$ for each $k\in \Z$ and
    \[\|A\|_{\mathcal{W}_{\alpha_1, \alpha_2}(\mathcal{A})} := \sum_{k\in \Z}\alpha(k)\|D_k(A)\|_\mathcal{A} <\infty.\]
Recalling \eqref{def.D(A)} note that $D_k(A)$ stands for the $k$-th diagonal of 
$A$, a sequence in $l^\infty(\Z)$. 
It is easy to see that $A\in\mathcal{W}_{\alpha_1, \alpha_2}(\mathcal{A})$ if and only if it
can be written as
    \begin{equation}\label{A-series}
    A = \sum_{k\in \Z}(a^{(k)}I)U_k
    \end{equation}
with $a^{(k)}\in\cA$ ($k\in\Z$) such that 
$$
\|A\|_{\mathcal{W}_{\alpha_1, \alpha_2}(\cA)}=\sum_{k\in\Z} \alpha(k) \| a^{(k)}\|_{\cA} <\infty.
$$
Note that  \eqref{A-series} converges absolutely both in the norm of $\mathcal{W}_{\alpha_1,\alpha_2}(\cA)$ and in the operator norm.  
Clearly, $\mathcal{W}_{\alpha_1,\alpha_2}(\cA)$ is a Banach space.
In fact, we have the following results.

\begin{prop}\label{WA main}
Let $\alpha_1,\alpha_2 \ge 0$, and let $\mathcal{A}$ be a shift-invariant and continuously embedded Banach subalgebra of $l^\infty(\Z)$. Then $\mathcal{W}_{\alpha_1, \alpha_2}(\mathcal{A})$ is a shift-invariant and continuously embedded Banach subalgebra of $\mathcal{L}(l^2(\Z))$. Moreover, \hfill
\begin{enumerate}[label=(\alph*)]
\item 
if $\alpha_1,\alpha_2 \ge0$, $\alpha_1 + \alpha_2 = 1$, and if $\|a\|_{l^\iy(\Z)}\le \|a\|_{\cA}$ for all $a\in \cA$,
 then $\mathcal{W}_{\alpha_1, \alpha_2}(\mathcal{A})$ is suitable, 
\item if $\mathcal{A}\subseteq AP(\Z)$, then $\mathcal{W}_{\alpha_1, \alpha_2}(\mathcal{A})$ is rigid.
\end{enumerate}
\end{prop}

For a detailed proof we refer to  \cite[Thm.~3.5]{ERS}. The cases of $\alpha_1=0$ or $\alpha_2=0$ in (a) were not explicitly stated there, but the proof proceeds along the same lines.
To eleborate a little bit more, let us remark that the proof of (a) is based on the fact that  $\mathcal{W}_{\alpha_1, \alpha_2}(\mathcal{A})$ 
is continuously embedded in $\cR_{p,q}$ as defined in Example \ref{ex3.1} with $p=1/\alpha_2$ and $q=1/\alpha_1$.
Part (b) follows from the fact that the class $\mathcal{OAP}$ is rigid. This was proved in \cite[Thm.~3.1]{ERS}
and uses the property that for each $a\in AP(\Z)$ there exists a strictly increasing integer sequence $k:\Zp\to\Z$
such that $U_{-k(n)}a\to a$ in $l^\infty(\Z)$  as $n\to\infty$.

\medskip
Concerning the notion of fractal sequences the following result holds.

\begin{prop}\label{pro eqn for fractal R}
Let $\alpha_1,\alpha_2 \ge 0$, and let $\mathcal{A}$ be a shift-invariant Banach subalgebra of $l^\infty(\Z)$. Then
\begin{itemize}
\item[(a)]
a sequence $h:\Zp\to\Z$ is fractal for $\mathcal{R} = \mathcal{W}_{\alpha_1, \alpha_2}(\mathcal{A})$ if and only if $h$ is fractal for $\mathcal{A}$.
\item[(b)]
If this is true, and if $h$ is fractal for $\mathcal{A}$ with associated $U:\cA\to\cA$, then $h$ is fractal for $\mathcal{R}$ with the associated $\mathcal{U}$ given by
\begin{equation}\label{frac for R}
\mathcal{U}:  \sum_{k\in\Z} (a^{(k)}I)U_k
\;\;\mapsto\;\; \sum_{k\in\Z} (U(a^{(k)})I)U_k. 
\end{equation}
Therein, the operator is defined for finite linear combinations and extends by continuity to all of $\mathcal{R}$.
\end{itemize}
\end{prop}
\begin{proof}
To show the `only if' part of (a), assume that $h$ is fractal for $\mathcal{R}$ with associated $\mathcal{U}$. 
Considering the special case $A = aI$ with $a\in \mathcal{A}$ we have
    \[\|U_{-h(n)}AU_{h(n)}- \mathcal{U}A\|_{\mathcal{R}} = \|(U_{-h(n)}a)I - \mathcal{U}(aI))\|_{\mathcal{R}}\to 0\]
as $n\to\infty$. Therefore, $\mathcal{U}(aI) = U(a)I$ for some $U(a)\in \mathcal{A}$, and $\|U_{-h(n)}a - U(a)\|_\mathcal{A}\to 0$ as $n\to \infty$. Hence, if we identify $\mathcal{A}$ with $\{aI:a\in\mathcal{A}\}$, then $U:\mathcal{A}\to \mathcal{A}$ is the restriction map of $\mathcal{U}$ onto $\mathcal{A}$ and (\ref{eqn for fractal}) is satisfied.
Therefore, $h$ is fractal for $\mathcal{A}$.

Let us now show (b) and the `if' part of (a).
Assume that $h$ is fractal for $\mathcal{A}$ with associated $U$.  We are going to show that $h$ is fractal for $\mathcal{R}$ with the associated $\mathcal{U}$ given above. Firstly consider those $A\in \mathcal{R}$ which can be written as a finite linear combination $A = \sum_k (a^{(k)}I)U_k$, $a^{(k)}\in \mathcal{A}$. 
Observing that 
$$
U_{-h(n)}A U_{h(n)} = \sum_{k} ((U_{-h(n)}a^{(k)})I)U_k
\;\; \to\;\; 
\sum_k (U(a^{(k)})I)U_k  = \mathcal{U}A
$$
as $n\to\infty$, where the convergence is in the norm of $\mathcal{R}$. 
In other words, we have strong convergence of the bounded linear operators $\mathcal{U}^{h(n)}$ defined by
$A\in \mathcal{R}\mapsto U_{-h(n)}AU_{h(n)} \in \mathcal{R}$ on the dense subset of $\mathcal{R}$. In addition, the norm of $U^{h(n)}$ equals one because of shift-invariance, and we have strong convergence of $\mathcal{U}^{h(n)}$ to $\mathcal{U}$ on all of $\mathcal{R}$ followed by a standard approximation argument. It follows that the operator $\mathcal{U}$ extends by continuity to all of $\mathcal{R}$ and $\|\mathcal{U}A \|_\mathcal{R} = \|A\|_{\mathcal{R}}$ for all $A\in \mathcal{R}$. 
\end{proof}

Besides establishing an abstract version of a Szeg\"{o}-Widom limit theorem in Section \ref{sec:5}, we will specialize it to
the Banach algebras $\mathcal{R} = \mathcal{W}_{\alpha_1, \alpha_2}(\mathcal{A})$ with $\mathcal{A} = APW(\Z, \Xi, \beta)$, where $\Xi$ and $\beta$ have to satisfy certain conditions. Apart from Theorem \ref{thm trace eva}, we
will only need the following summary of statements resulting from  Propositions \ref{prop fractal A}, \ref{p2.2},
\ref{WA main}, and \ref{pro eqn for fractal R}.

\begin{cor}\label{c3.4}
Let $\beta$ be an admissible weight on an additive subgroup $\Xi$ of $\R/\Z$, and let $\alpha_1,\alpha_2\ge0$ such that 
$\alpha_1+\alpha_2=1$. 
\begin{itemize}
\item[(i)]
$\cR= \mathcal{W}_{\alpha_1, \alpha_2}(\mathcal{A})$ with $\mathcal{A} = APW(\Z, \Xi, \beta)$
is a rigid, suitable, shift-invariant and unital Banach algebra of bounded linear operators on $l^2(\Z)$.
\item[(ii)]
If $h:\Zp\to\Z$ is a fractal sequence for $\Xi$ with associated 
$\tau\in \HomXT$, then $h$ is fractal for $\cR$ with associated $\cU$ given by
\begin{equation}\label{mapU}
\cU:\sum_{n\in\Z} \sum_{\xi\in \Xi} a^{(n)}_\xi (e_\xi I)U_n \mapsto
\sum_{n\in\Z} \sum_{\xi\in \Xi} a^{(n)}_\xi \tau(\xi) (e_\xi I)U_n.
\end{equation}
\end{itemize}
\end{cor}

\section{\Large\bf Banach algebras associated with suitable Banach \\  algebras}
\label{sec:4}

In this section, we will continue to prepare the proof of the abstract version of the limit theorem in the next section. We construct three types of Banach algebras which are naturally associated with any suitable, rigid, shift-invariant and unital
 Banach algebra $\mathcal{R}$.

\subsection{Banach algebras of operators on $\boldsymbol{l^2(\Zp)}$}
\label{sec:4.1}

The first two are the Banach algebras  $\mathcal{O}(\mathcal{R})$ and $\mathcal{O}(\widetilde{\mathcal{R}})$.
They were already considered in  \cite[Sect.~4.1]{ERS}, and it suffices to cite the corresponding result from there.
Recall that $\widetilde{A}=JAJ$.

\begin{prop}
Let $\mathcal{R}$ be a rigid, suitable and unital Banach subalgebra of $\mathcal{L}(l^2(\Z))$.
\begin{itemize}
\item[(a)] 
The set
    \begin{equation}
        \mathcal{O}(\mathcal{R}):=\{T(A) + K:A\in \mathcal{R}, K\in \mathcal{C}_1(l^2(\Zp))\}
    \end{equation}
is a unital Banach algebra with the norm
    \begin{equation}
        \|T(A)+K\|_{\mathcal{O}(\mathcal{R})}:= \|A\|_\mathcal{R} + \|K\|_{\mathcal{C}_1}.
    \end{equation}
\item[(b)] 
The set
    \begin{equation}
        \mathcal{O}(\widetilde{\mathcal{R}}) := \{T(\widetilde{A})+K:A\in \mathcal{R}, K\in \mathcal{C}_1(l^2(\Zp))\}
    \end{equation}
is a unital Banach algebra with the norm
    \begin{equation}
       \|T(\widetilde{A})+K\|_{\mathcal{O}(\widetilde{\mathcal{R}})}: = \|A\|_\mathcal{R} + \|K\|_{\mathcal{C}_1}.
    \end{equation}
\end{itemize}
\end{prop}

We also need the following two results, which were established in  \cite[Sect.~3.3]{ERS}.

\begin{prop}\label{P plus trace class}
Let $\mathcal{R}$ be a suitable and unital Banach algebra, and let $A_1, \dots, A_r\in \mathcal{R}$. Then the functions
    \begin{align*}
        F_0(\lambda_1,\dots,\lambda_r)&:= T(e^{\lambda_1A_1}\cdots e^{\lambda_rA_r})e^{-\lambda_rT(A_r)}\cdots e^{-\lambda_1T(A_1)} - P,\\
        F_1(\lambda_1,\dots,\lambda_r)&:= T(e^{\lambda_1\widetilde{A_1}}\cdots e^{\lambda_r\widetilde{A_r}})e^{-\lambda_rT(\widetilde{A_r})}\cdots e^{-\lambda_1T(\widetilde{A_1})} - P
    \end{align*}
are analytic with respect to each variable $\lambda_k\in \mathbb{C}$ and take values in $\mathcal{C}_1(l^2(\Zp))$.
\end{prop}

As a consequence, the  operator determinants
$$
        \det T(e^{\lambda_1A_1}\cdots e^{\lambda_rA_r})e^{-\lambda_rT(A_r)}\cdots e^{-\lambda_1T(A_1)},
$$
and 
$$
     \det T(e^{\lambda_1\widetilde{A_1}}\cdots e^{\lambda_r\widetilde{A_r}})e^{-\lambda_rT(\widetilde{A_r})}\cdots e^{-\lambda_1T(\widetilde{A_1})}
$$
are well-defined and depend analytically on each of the complex variables $\lambda_k$.
For the definition and basic properties of operator determinants, see, e.g.,  \cite{gohberg2013classes}.
As it turns out, these two kinds of operator determinants are related to each other.

\begin{prop}\label{prop det iden}
Let $\mathcal{R}$ be a suitable and unital Banach algebra, and let $A_1, \dots, A_r\in \mathcal{R}$. Then for each $\lambda_1, \dots, \lambda_r\in \mathbb{C}$, the operator determinant
    \[f(\lambda_1, \dots, \lambda_r):= \det T(e^{\lambda_1\widetilde{A_1}}\cdots e^{\lambda_r\widetilde{A_r}})e^{-\lambda_rT(\widetilde{A_r})}\cdots e^{-\lambda_1T(\widetilde{A_1})}\]
is equal to the operator determinant
    \[g(\lambda_1, \dots, \lambda_r):= \det e^{\lambda_1T(A_1)}\cdots e^{\lambda_rT(A_r)}T(e^{-\lambda_rA_r}\cdots e^{-\lambda_1A_1}).\]
\end{prop}

\subsection{Banach algebras of sequences of finite sections}
\label{sec:4.2}

The third kind of Banach algebras associated with $\mathcal{R}$ are Banach algebras $\mathcal{S}_{h_1,h_2}(\cR)$ of sequences of matrices. In contrast to \cite[Sect.~4.2]{ERS} we need to consider a more general situation, which involves fractal sequences $h_1$ and $h_2$ for $\mathcal{R}$ (instead of distinguished sequences) and is motivated by the following.

\medskip
Let $h_1,h_2:\Zp\to\Z$ be two fractal sequences for $\mathcal{R}$ such that $h(n):=h_2(n) - h_1(n)>0$ and $h(n)\to \infty$
as $n\to\iy$.
Then, for $A\in\mathcal{R}\subseteq \mathcal{L}(l^2(\Z))$,
    \[P_{h_1(n), h_2(n)}AP_{h_1(n), h_2(n)}= U_{h_1(n)}P_{h(n)}P(U_{-h_1(n)}AU_{h_1(n)})PP_{h(n)}U_{-h_1(n)},\]
i.e., the sequence of finite sections obtained through $P_{h_1(n),h_2(n)}$ can be identified with a sequence of $h(n)\times h(n)$ matrices. Using the notation defined in \eqref{Un}
we obtain
\begin{align}\label{det-2}
\det P_{h_1(n), h_2(n)}AP_{h_1(n), h_2(n)}  
&= 
\det P_{h(n)} T(\mathcal{U}^{h_1(n)}A)P_{h(n)},
\\
 \label{tr-2}
 \trace P_{h_1(n), h_2(n)}AP_{h_1(n), h_2(n)}                              
 &= 
 \trace P_{h(n)} T(\mathcal{U}^{h_1(n)}A)P_{h(n)}.            
\end{align}
Therefore, as far as traces and determinants are concerned, instead of the finite sections
$P_{h_1(n),h_2(n)}A  P_{h_1(n),h_2(n)}$ it suffices to consider $P_{h(n)} T(\mathcal{U}^{h_1(n)}A)P_{h(n)}$.
In the following theorem we define the Banach algebra $\mathcal{S}_{h_1, h_2}(\mathcal{R})$,
the elements of which include such sequences.

\medskip
We also need the reflection operator $W_n: l^2(\Zp)\to l^2(\Zp)$, $n\in \Zp$, defined by
\begin{equation}\label{Wn} 
W_n:(x_k)_{k\in \Zp} \mapsto (y_k)_{k\in \Zp},\quad
        y_k = \begin{cases}
            x_{n-1-k} &\mathrm{ if }\enskip  0\le k<n\\
            0  &\mathrm{ if } \enskip k\ge n.
        \end{cases}
\end{equation}
Note that $W_n^2=P_n$ and $\mathrm{im} \,W_n=\mathrm{im}\, P_n$.

\begin{thm}\label{thm:4.4}
Let $\mathcal{R}$ be a rigid, suitable, shift-invariant and unital Banach subalgebra of $\mathcal{L}(l^2(\Z))$, and let $h_1$ and $h_2$ be  fractal sequences for $\mathcal{R}$ with associated $\mathcal{U}_1$ and $\mathcal{U}_2$, respectively, such that $h(n) := h_2(n) - h_1(n)>0$ and $h(n)\to \infty$ as $n\to\infty$. 

Then the set $\mathcal{S}_{h_1, h_2}(\mathcal{R})$ consisting of all sequences $(A_n)_{n\ge1}$ of operators $A_n:\mathrm{im} P_{h(n)}\to  \mathrm{im} P_{h(n)}$ of the form
    \begin{equation}\label{An}
        A_n = P_{h(n)}T(\mathcal{U}^{h_1(n)}A)P_{h(n)} + P_{h(n)}KP_{h(n)} + W_{h(n)}LW_{h(n)} + G_n
    \end{equation}
with $A\in \mathcal{R}$, $K, L\in \mathcal{C}_1(l^2(\Zp))$, $G_n\in \mathcal{C}_1(\mathrm{im} P_{h(n)})$ and $\|G_n\|_{\mathcal{C}_1}\to 0$ forms a unital Banach algebra with respect to the operations
    \[(A_n)+(B_n)=(A_n+B_n), \quad (A_n)\cdot(B_n) = (A_nB_n),\quad \lambda(A_n)= (\lambda A_n)\]
and with the norm
    \begin{equation}\label{Ann}
        \|(A_n)_{n\ge1}\|_{\mathcal{S}_{h_1, h_2}(\mathcal{R})}:= \|A\|_\mathcal{R} + \|K\|_{\mathcal{C}_1} + \|L\|_{\mathcal{C}_1} + \sup_{n\ge 1} \|G_n\|_{\mathcal{C}_1}.
    \end{equation}
Moreover, the set $\mathcal{J}_h(\mathcal{R})$ of all sequences $(J_n)$ of the form
    \[J_n = P_{h(n)}KP_{h(n)} + W_{h(n)}LW_{h(n)} + G_n\]
with $K, L\in \mathcal{C}_1(l^2(\Zp))$, $G_n\in \mathcal{C}_1(\mathrm{im} P_{h(n)})$and $\|G_n\|_{\mathcal{C}_1}\to 0$ forms a closed two-sided ideal of $\mathcal{S}_{h_1, h_2}(\mathcal{R})$.
\end{thm}
\begin{proof}
Since $h_1$ and $h_2$ are fractal sequences for $\mathcal{R}$ with associated isometric Banach algebra homomorphisms $\mathcal{U}_1$ and $\mathcal{U}_2$
on $\mathcal{R}$,  we have for any $A\in \mathcal{R}$ that
    \begin{equation}\label{eq:4.8}
    \lim_{n\to \infty} \|\mathcal{U}^{h_1(n)}A - \mathcal{U}_1A\|_\mathcal{R} = 0
    \quad\mbox{and}\quad
    \lim_{n\to \infty} \|\mathcal{U}^{h_2(n)}A - \mathcal{U}_2A\|_\mathcal{R} = 0.
    \end{equation}
    Note that this implies, in particular, convergence in the operator norm.
    
Let us first show that the norm \eqref{Ann} is well-defined. For this we need to show that the various terms on the right hand side of 
\eqref{An} are uniquely determined  by 
the sequence $(A_n)$. Therefore, consider a sequence
    \[P_{h(n)}T(\mathcal{U}^{h_1(n)}A)P_{h(n)} + P_{h(n)}KP_{h(n)} + W_{h(n)}LW_{h(n)} + G_n = 0.\]
Since $W_{h(n)}\to0$ weakly, taking the strong limit as $n\to \infty$ yields that $T(\mathcal{U}_1A)+ K = 0$. Because of rigidity we have $K = 0$ and $\mathcal{U}_1A = 0$, which implies that $A=0$ as well. It follows that 
    \[W_{h(n)}LW_{h(n)} + G_n = 0.\]
Multiplying with $W_{h(n)}$ from both sides and noting that $W_{h(n)}^2=P_{h(n)}$, we obtain $L= 0$ by taking again strong limits. Thus 
$G_n=0$ as well, and this shows that the norm is well-defined.

It is obvious that $\mathcal{S}_{h_1, h_2}(\mathcal{R})$ is a linear space, which is complete with respect to its norm. 
It remains to show that it is indeed a Banach algebra. Consider
    \begin{align}
    \label{f:An}
         A_n &= P_{h(n)}T(\mathcal{U}^{h_1(n)}A)P_{h(n)} + P_{h(n)}K_1P_{h(n)} + W_{h(n)}L_1W_{h(n)} + G_n^{(1)},\\
     \label{f:Bn}
         B_n &= P_{h(n)}T(\mathcal{U}^{h_1(n)}B)P_{h(n)} + P_{h(n)}K_2P_{h(n)} + W_{h(n)}L_2W_{h(n)} + G_n^{(2)}.
    \end{align}
We have to show that $(A_nB_n)\in \mathcal{S}_{h_1, h_2}(\mathcal{R})$ and 
	\[\|(A_nB_n)\|_{\mathcal{S}_{h_1, h_2}(\mathcal{R})} \le C\,
	\|(A_n)\|_{\mathcal{S}_{h_1, h_2}(\mathcal{R})}
	\|(B_n)\|_{\mathcal{S}_{h_1, h_2}(\mathcal{R})}\]
for some constant $C$. For that purpose, we multiply each term in the first sum with each term in the second sum and are led to consider several cases. Note that if one of the factors is $G_n^{(1)}$ or $G_n^{(2)}$, then the product can easily be taken care of.

The remaining products are dealt with as follows. Firstly,
    \begin{align*}
        P_{h(n)}K_1P_{h(n)}\cdot P_{h(n)}K_2P_{h(n)} &= P_{h(n)}K_1K_2P_{h(n)} - P_{h(n)}K_1Q_{h(n)}K_2P_{h(n)},\\
        W_{h(n)}L_1W_{h(n)}\cdot W_{h(n)}L_2W_{h(n)} &= W_{h(n)}L_1L_2W_{h(n)} - W_{h(n)}L_1Q_{h(n)}L_2W_{h(n)},\\
        P_{h(n)}K_1P_{h(n)}\cdot W_{h(n)}L_2W_{h(n)} &= P_{h(n)}K_1W_{h(n)}L_2W_{h(n)},\\
	W_{h(n)}L_1W_{h(n)}\cdot P_{h(n)}K_2P_{h(n)} &= W_{h(n)}L_1W_{h(n)}K_2P_{h(n)}.
    \end{align*}
Since $Q_{h(n)}:= P-P_{h(n)}$ converges to zero strongly, and $W_{h(n)}$ converges to zero weakly on $l^2(\Zp)$, the last term in the first two equations as well as the terms in the third and fourth equation converge to zero in the trace norm.

Now consider the following product and write it as
    \begin{align*}
                 \lefteqn{P_{h(n)}T(\mathcal{U}^{h_1(n)}A)P_{h(n)}\cdot P_{h(n)}K_2P_{h(n)}} \\
        &= P_{h(n)}T(\mathcal{U}^{h_1(n)}A)K_2P_{h(n)} - P_{h(n)}T(\mathcal{U}^{h_1(n)}A)Q_{h(n)}K_2P_{h(n)}\\
        &= P_{h(n)}T(\mathcal{U}_1A)K_2P_{h(n)} + P_{h(n)}T(\mathcal{U}^{h_1(n)}A - \mathcal{U}_1A)K_2P_{h(n)} - P_{h(n)}T(\mathcal{U}^{h_1(n)}A)Q_{h(n)}K_2P_{h(n)}\\
        &= P_{h(n)}T(\mathcal{U}_1A)K_2P_{h(n)}+C_n',
    \end{align*}
where $C_n'$ consist of the last two terms, both of which converge to zero in the trace norm due to 
 \eqref{eq:4.8} and since $Q_n=Q_n^*\to0$ strongly.
Similarly,
	\begin{align*}
	  \lefteqn{P_{h(n)}K_1P_{h(n)}\cdot P_{h(n)}T(\mathcal{U}^{h_1(n)}B)P_{h(n)} }\\
	&=  P_{h(n)}K_1T(\mathcal{U}_1B)P_{h(n)} +  P_{h(n)}K_1T(\mathcal{U}^{h_1(n)}B - \mathcal{U}_1B)P_{h(n)} - P_{h(n)}K_1Q_{h(n)}T(\mathcal{U}^{h_1(n)}B)P_{h(n)}\\
	&=  P_{h(n)}K_1T(\mathcal{U}_1B)P_{h(n)}+C_n'',
	\end{align*}
where $C_n''$ converges to zero in the trace norm for the same reasons. Morevover, observe that we have the estimates 
$$
\|C_n'\|_{\cC_1}\le 3 \|A\|_{\cR} \|K_2\|_{\cC_1},\qquad \|C_n''\|_{\cC_1}\le 3 \|K_1\|_{\cC_1}\|B\|_{\cR}
$$
for all $n$.

Before proceeding to the next cases, remark that   
 \begin{equation}\label{eq:4.9}
 W_{h(n)}P = P_{h(n)}PJU_{-h(n)}\quad\mbox{and}\quad PW_{h(n)} = U_{h(n)}JPP_{h(n)},
 \end{equation}
and therefore
	\begin{align*}
		W_{h(n)}T(\mathcal{U}^{h_1(n)} A)W_{h(n)} &= P_{h(n)}PJU_{-h(n)}
		   (\mathcal{U}^{h_1(n)}A)    U_{h(n)}JPP_{h(n)}\\
				&= P_{h(n)} T(\widetilde{\mathcal{U}^{h_2(n)}A})P_{h(n)}. 
	\end{align*}
Here recall that $\widetilde{A}=JAJ$, that $\mathcal{U}^{h_i(n)}$ are defined in \eqref{Un}, and note 
that $h(n)=h_2(n)-h_1(n)$.
Using this identity, the following two products can be written as
\begin{align*}
        P_{h(n)}  T(\mathcal{U}^{h_1(n)}A)  P_{h(n)}\cdot W_{h(n)}L_2W_{h(n)}
         &= W_{h(n)}\left(P_{h(n)}  T(\widetilde{\mathcal{U}^{h_2(n)}A})  P_{h(n)})L_2P_{h(n)} \right) W_{h(n)},
  \\
  W_{h(n)}L_1W_{h(n)} \cdot  P_{h(n)}  T(\mathcal{U}^{h_1(n)}B)  P_{h(n)}
  &= W_{h(n)} \left( P_{h(n)} L_1   P_{h(n)} T(\widetilde{\mathcal{U}^{h_2(n)} B }) P_{h(n)} \right) W_{h(n)}.
\end{align*}
The resulting expressions are of the same form as those in the last two cases
except for the $W_{h(n)}$'s on each side. One can deal with them in the same manner to obtain
\begin{align*}
        P_{h(n)}  T(\mathcal{U}^{h_1(n)}A)  P_{h(n)}\cdot W_{h(n)}L_2W_{h(n)}
         &= W_{h(n)} T(\widetilde{\mathcal{U}_2 A}) L_2 W_{h(n)} + D_n', 
  \\
  W_{h(n)}L_1W_{h(n)} \cdot  P_{h(n)}  T(\mathcal{U}^{h_1(n)}B)  P_{h(n)}
  &= W_{h(n)} L_1    T(\widetilde{\mathcal{U}_2 B }) W_{h(n)} + D_n'',
\end{align*}
where $D_n'$ and $D_n''$ are sequences tending to zero in the trace norm, while in addition
their trace norms are bounded.

Finally, the last case to consider is that of the product
    \begin{align*}
            \lefteqn{ P_{h(n)}T(\mathcal{U}^{h_1(n)}A)P_{h(n)}\cdot P_{h(n)}T(\mathcal{U}^{h_1(n)}B)P_{h(n)}}\qquad
            \\
          & =
          P_{h(n)}T(\mathcal{U}^{h_1(n)}(AB))P_{h(n)}-P_{h(n)}H(\mathcal{U}^{h_1(n)}A)H(\widetilde{\mathcal{U}^{h_1(n)}B})P_{h(n)} \\
         &\quad  - P_{h(n)}T(\mathcal{U}^{h_1(n)}A)Q_{h(n)}T(\mathcal{U}^{h_1(n)}B)P_{h(n)}.
    \end{align*}
Here we used  \eqref{T-id} and that $\mathcal{U}^{h_1(n)}$ is multiplicative. The second term $H(\mathcal{U}^{h_1(n)}A)H(\widetilde{\mathcal{U}^{h_1(n)}B})$ is trace class
and we have 
$$
\|H(\mathcal{U}^{h_1(n)}A)H(\widetilde{\mathcal{U}^{h_1(n)}B})\|_{\cC_1}\le M \| \mathcal{U}^{h_1(n)}A\|_{\cR}
\|\mathcal{U}^{h_1(n)}B\|_{\cR} = M \|A\|_{\cR}\|B\|_{\cR}
$$
 by (\ref{suitable}) and since $\cR$ is shift-invariant. In fact, we can write
\begin{align*}
\lefteqn{H(\mathcal{U}^{h_1(n)}A)H(\widetilde{\mathcal{U}^{h_1(n)}B})}
 \\
 &= 
 H(\mathcal{U}_1A)H(\widetilde{\mathcal{U}_1 B})
 +
 H(\mathcal{U}^{h_1(n)}A-\mathcal{U}_1A)H(\widetilde{\mathcal{U}_1B})
 +
 H(\mathcal{U}^{h_1(n)}A)H(\widetilde{\mathcal{U}^{h_1(n)}B-\mathcal{U}_1B}).
\end{align*}
Using \eqref{eq:4.8} and the estimate of the kind we just employed, it is easily seen that 
the last two terms converge to zero in the trace norm as $n\to\infty$.
Their trace norm can be estimated uniformly as well. Therefore, the second term becomes
$$
P_{h(n)} H(\mathcal{U}^{h_1(n)}A)H(\widetilde{\mathcal{U}^{h_1(n)}B}) P_{h(n)}
 =  P_{h(n)} H(\mathcal{U}_1 A) H(\widetilde{\mathcal{U}_1 B  }) P_{h(n)}+ E_n'
$$ 
with $ H(\mathcal{U}_1 A) H(\widetilde{\mathcal{U}_1 B  })$ being trace class, and $E_n'\to0$ in trace norm, 
and $\|E_n'\|\le 2M\|A\|_{\cR}\|B\|_{\cR}$. Regarding the third term we note that 
$$
P_{h(n)}T(\mathcal{U}^{h_1(n)}A)Q_{h(n)}T(\mathcal{U}^{h_1(n)}B)P_{h(n)}
= W_{h(n)} H(\widetilde{ \mathcal{U}^{h_2(n)} A}) H(\mathcal{U}^{h_2(n)} B) W_{h(n)}.
$$ 
Indeed, this identity can be verified by using \eqref{eq:4.9} and  the identity $PQ_{h(n)}P = U_{h(n)}PU_{-h(n)}$.
With the same kind of arguments as just  employed it follows that this equals 
$$
W_{h(n)} H(\widetilde{ \mathcal{U}_2A})H(\mathcal{U}_2  B) W_{h(n)} + E_n''
$$
with $H(\widetilde{ \mathcal{U}_2A})H(\mathcal{U}_2  B)$ being trace class, and $E_n''\to0$ in trace norm, 
and $\|E_n''\|\le 2M\|A\|_{\cR}\|B\|_{\cR}$.  We now arrive at 
\begin{align*}
       \lefteqn{ P_{h(n)}T(\mathcal{U}^{h_1(n)}A)P_{h(n)}\cdot P_{h(n)}T(\mathcal{U}^{h_1(n)}B)P_{h(n)}}\qquad
            \\
          & =
          P_{h(n)}T(\mathcal{U}^{h_1(n)}(AB))P_{h(n)}-P_{h(n)}H(\mathcal{U}_1A)H(\widetilde{\mathcal{U}_1B})P_{h(n)} \\
         &\quad  - P_{h(n)}H(\widetilde{\mathcal{U}_2A})H(\mathcal{U}_2B)W_{h(n)}-E_n'-E_n''.   
\end{align*}

To summarize what we have done so far, the product of the above $(A_n)$ and $(B_n)$ can be written as
    \begin{equation}\label{AnBn}
        A_nB_n = P_{h(n)}T\left(\mathcal{U}^{h_1(n)}(AB)\right)P_{h(n)} + P_{h(n)}KP_{h(n)} + W_{h(n)}LW_{h(n)} + G_n,
    \end{equation}
where
    \begin{align}
        K &= K_1K_2 + T(\mathcal{U}_1A)K_2 + K_1 T(\mathcal{U}_1B) - H(\mathcal{U}_1A)H(\widetilde{\mathcal{U}_1B}),\\
        L &= L_1L_2 + T(\widetilde{\mathcal{U}_2A})L_2 + L_1T(\widetilde{\mathcal{U}_2B}) - H(\widetilde{\mathcal{U}_2A})H(\mathcal{U}_2B)\label{AnBn KL}
    \end{align}
and where $G_n$ is a sequence of trace class operators converging to zero in the trace norm. Expressing $G_n$ explicitly using the above computations, it is easily seen that there exists a constant $C$ such that
    \[\|(A_nB_n)\|_{\mathcal{S}_{h_1, h_2}(\mathcal{R})} \le C\,\|(A_n)\|_{\mathcal{S}_{h_1, h_2}(\mathcal{R})}
    \|(B_n)\|_{\mathcal{S}_{h_1, h_2}(\mathcal{R})}.\]
Therefore, we can conclude that $\mathcal{S}_{h_1,h_2}(\cR)$ is a Banach algebra.
    
Finally, formulas (\ref{AnBn})--(\ref{AnBn KL}) imply that $\mathcal{J}_h(\mathcal{R})$ is an ideal of $\mathcal{S}_{h_1, h_2}(\mathcal{R})$, and it is indeed closed by the definition of the norm.
\end{proof}

\begin{thm}\label{thm:4.5}
Under the assumptions of the preceding theorem, the mappings $\mathcal{W}_{h_1,h_2}$ and $\widetilde{\mathcal{W}_{h_1,h_2}}$ defined by
    \begin{align}
        \mathcal{W}_{h_1,h_2}:\quad (A_n)_{n\ge 1}\in \mathcal{S}_{h_1, h_2}(\mathcal{R}) &\quad\mapsto\quad T(\mathcal{U}_1A) + K\in \mathcal{O}(\mathcal{R}),\\
        \widetilde{\mathcal{W}_{h_1,h_2}}:\quad  (A_n)_{n\ge 1}\in \mathcal{S}_{h_1, h_2}(\mathcal{R}) &\quad \mapsto\quad T(\widetilde{\mathcal{U}_2A}) + L\in \mathcal{O}(\widetilde{\mathcal{R}}),
    \end{align}
where $(A_n)$ is of the form (\ref{An}), are well defined unital Banach algebra homomorphisms.
\end{thm}

Here, as before,  $\mathcal{U}_1$ and $ \mathcal{U}_2$ are the (isometric, unital) Banach algebra homomorphisms on 
$\mathcal{R}$ associated with the fractal sequences $h_1$ and $h_2$.

\begin{proof}
The norms of these mappings applied to a sequence $(A_n)$ of the form \eqref{An} are given by
    \begin{align*}
        \|\mathcal{W}_{h_1,h_2}((A_n))\|_{\mathcal{O}(\mathcal{R})} &= \|\mathcal{U}_1A\|_\mathcal{R} + \|K\|_{\mathcal{C}_1} = \|A\|_\mathcal{R} + \|K\|_{\mathcal{C}_1},\\
        \|\widetilde{\mathcal{W}_{h_1,h_2}}((A_n))\|_{\mathcal{O}(\widetilde{\mathcal{R}})} &= \|\widetilde{\mathcal{U}_2A}\|_{\widetilde{\mathcal{R}}} + \|L\|_{\mathcal{C}_1} = \|A\|_\mathcal{R} + \|L\|_{\mathcal{C}_1}.
    \end{align*}
In view of \eqref{Ann} this implies that these mappings are well defined and continuous. Linearity is obvious. 
Multiplicativity follows from formulas (\ref{AnBn})--(\ref{AnBn KL}). For instance, if $(A_n)$ and $(B_n)$ are given by
\eqref{f:An} and \eqref{f:Bn}, then
$$
\mathcal{W}_{h_1,h_2}((A_n))= T(\mathcal{U}_1A)+K_1,\qquad
\mathcal{W}_{h_1,h_2}((B_n))= T(\mathcal{U}_1B)+K_2,
$$
while 
$$
\mathcal{W}_{h_1,h_2}((A_n)(B_n)) =
T(\mathcal{U}_1(AB))+K_1K_2+T(\mathcal{U}_1A)K_2+K_1 T(\mathcal{U}_1B)-H(\mathcal{U}_1A)
H(\widetilde{\mathcal{U}_1 B}).
$$
Now take \eqref{T-id} into account and note that $\mathcal{U}_1$ is multiplicative as well.
\end{proof}

\section{\Large \bf First versions of the limit theorem}
\label{sec:5}

In this section we establish the abstract version and a more concrete version of the limit theorems.
Both deal with the asymptotics of the determinants of  $P_{h_1(n),h_2(n)}A P_{h_1(n),h_2(n)}$ where $h_1$ and $h_2$ are fractal sequences.

The abstract version reduces the determinant asymptotics to the asymptotics of traces. It applies to operators $A$ taken from any suitable, rigid, shift-invariant and unital Banach algebra $\cR$. 

In the more concrete version (referred to as the fractal version), we specialize to the case of 
the Banach algebras $\cR=\mathcal{W}_{\alpha_1,\alpha_2}(\cA)$ with  $\cA=APW(\Z,\Xi,\beta)$.
 Here the trace computation is carried out by help of Theorem \ref{thm trace eva}.

\medskip
We need two more auxiliary results for the abstract version.
The first one is a direct consequence of the results of the previous section.

\begin{prop}\label{prop Bn}
Let $\mathcal{R}$ be a rigid, suitable, shift-invariant and unital Banach subalgebra of $\mathcal{L}(l^2(\Z))$,
 and let $h_1$ and $h_2$ be  fractal sequences for $\mathcal{R}$ with associated $\mathcal{U}_1$ and $\mathcal{U}_2$, respectively, such that $h(n) := h_2(n) - h_1(n)>0$ and $h(n)\to \infty$ as $n\to\infty$. 

Then for any $A_1, \dots, A_r\in \mathcal{R}$, the sequence $(B_n)$ defined by
    \begin{equation}\label{Bn for SW thm}
        B_n :=P_{h(n)}T(\mathcal{U}^{h_1(n)}(e^{A_1}\cdots e^{A_r}))P_{h(n)}\cdot e^{-P_{h(n)}T(\mathcal{U}^{h_1(n)}A_r)P_{h(n)}}\cdots e^{-P_{h(n)}T(\mathcal{U}^{h_1(n)}A_1)P_{h(n)}}
    \end{equation}
belongs to $\mathcal{S}_{h_1,h_2}(\mathcal{R})$. Furthermore, there exist $K,L\in \mathcal{C}_1(l^2(\Zp))$, and $G_n\in \mathcal{C}_1(\mathrm{im} P_{h(n)})$ with $\|G_n\|_{\mathcal{C}_1}\to 0$ such that
    \begin{equation}\label{Bn simp for SW thm}
        B_n = P_{h(n)} + P_{h(n)}KP_{h(n)} + W_{h(n)}LW_{h(n)} + G_n.
    \end{equation}
The operators $K$ and $L$ are determined by
    \begin{align}
        P+K &= T(\mathcal{U}_1(e^{A_1}\cdots e^{A_r}))e^{-T(\mathcal{U}_1A_r)}\cdots e^{-T(\mathcal{U}_1A_1)},\\
        P+L &= T(\widetilde{\mathcal{U}_2(e^{A_1}\cdots e^{A_r})})e^{-T(\widetilde{\mathcal{U}_2A_r})}\cdots e^{-T(\widetilde{\mathcal{U}_2A_1})}.
    \end{align}
\end{prop}
\begin{proof}
Note that for any $(A_n)\in \mathcal{S}_{h_1,h_2}(\mathcal{R})$, $(e^{A_n}) = e^{(A_n)}$, and hence $(B_n)$ defined above is in $\mathcal{S}_{h_1,h_2}(\mathcal{R})$. Define the bounded linear map
    \begin{equation}
        \Lambda: A\in \mathcal{R} \enskip \mapsto \enskip (P_{h(n)}T(\mathcal{U}^{h_1(n)}A)P_{h(n)})\in \mathcal{S}_{h_1,h_2}(\mathcal{R}).
    \end{equation}
We rewrite $(B_n)$ as
    \[(B_n) = \Lambda(e^{A_1}\cdots e^{A_r})e^{-\Lambda(A_r)}\cdots e^{-\Lambda(A_1)}.\]
Furthermore, denote by $\Phi$ the natural homomorphism
    \[\Phi:\mathcal{S}_{h_1,h_2}(\mathcal{R}) \to \mathcal{S}_{h_1,h_2}(\mathcal{R})/\mathcal{J}_h(\mathcal{R}),\quad (A_n)\mapsto (A_n) +\mathcal{J}_h(\mathcal{R}). \]
Then $\Phi\circ \Lambda:\mathcal{R}\to \mathcal{S}_{h_1,h_2}(\mathcal{R})/\mathcal{J}_h(\mathcal{R})$ is a continuous Banach algebra homomorphism. Indeed, in view of \eqref{f:An}, \eqref{f:Bn}, and \eqref{AnBn} we have
    \begin{align*}
        (\Phi\circ \Lambda) (A) (\Phi\circ \Lambda) (B) &= (P_{h(n)}T(\mathcal{U}^{h_1(n)}A)P_{h(n)})\cdot (P_{h(n)}T(\mathcal{U}^{h_1(n)}B)P_{h(n)}) 
											+ \mathcal{J}_h(\mathcal{R})\\
                                                                            &= (P_{h(n)}T(\mathcal{U}^{h_1(n)}(AB))P_{h(n)}) + \mathcal{J}_h(\mathcal{R}) = (\Phi\circ \Lambda) (AB).
    \end{align*}
Applying the homomorphism $\Phi$ to $(B_n)$ yields
    \begin{align*}
        \Phi((B_n)) &= (\Phi\circ\Lambda)(e^{A_1}\cdots e^{A_r})e^{-(\Phi\circ\Lambda)(A_r)}\cdots e^{-(\Phi\circ\Lambda)(A_1)}\\
                  &= (\Phi\circ\Lambda)(e^{A_1}\cdots e^{A_r} e^{-A_r}\cdots e^{-A_1})\\
                  &= (\Phi\circ\Lambda)(I)\\
                  &= (P_{h(n)})+ \mathcal{J}_h(\mathcal{R}),
    \end{align*}
which shows $(B_n)$ is of the form (\ref{Bn simp for SW thm}). To determine the operators $K$ and $L$, we apply the 
Banach algebra homomorphisms $\mathcal{W}_{h_1,h_2}$ and $\widetilde{\mathcal{W}_{h_1,h_2}}$ to both sides of (\ref{Bn simp for SW thm}).
\end{proof}

The second auxilliary result is the following. Its (simple) proof can be found, e.g., in  Lemmas 9.1 and 9.3 of  \cite{Eh-SzW}.

\begin{lem}\label{det iden}
Let $A_n = P_n + P_nKP_n +W_nLW_n +C_n$ with
    \[K,L\in \mathcal{C}_1(l^2(\Zp))\text{ , } C_n\in \mathcal{C}_1(\mathrm{im} P_n) \text{ and } \|C_n\|_{\mathcal{C}_1}\to 0 \text{ as } n\to\infty.\]
Then
    \begin{equation}
        \lim_{n\to\infty} \det A_n = \det(P+K)\det(P+L).
    \end{equation}
\end{lem}

Now we are able to state the first of our main results.

\begin{thm}[Abstract version of the limit theorem] \label{MainResult-0}
Let $\mathcal{R}$ be a rigid, suitable, shift-invariant and unital Banach subalgebra of $\mathcal{L}(l^2(\Z))$,
 and let $h_1$ and $h_2$ be  fractal sequences for $\mathcal{R}$ with associated $\mathcal{U}_1$ and $\mathcal{U}_2$, respectively, such that $h(n) := h_2(n) - h_1(n)>0$ and $h(n)\to \infty$ as $n\to\infty$. 

If $A_1, \dots, A_r\in \mathcal{R}$ and
    $A = e^{A_1}\cdots e^{A_r}$, 
then
    \begin{equation}\label{f:5.7}
        \lim_{n\to\infty} \frac{\det(P_{h_1(n), h_2(n)}AP_{h_1(n), h_2(n)})}{\exp(\trace (P_{h_1(n), h_2(n)}(A_1+\cdots+A_r)P_{h_1(n), h_2(n)}))}
            = \det(\mathcal{B}_1)\det(\mathcal{B}_2),
    \end{equation}
where
    \begin{align}\label{A1A2}
        \mathcal{B}_1 &= T(\mathcal{U}_1A)\cdot e^{-T(\mathcal{U}_1A_r)}\cdots e^{-T(\mathcal{U}_1A_1)},\\
        \mathcal{B}_2 &= e^{T\left(\mathcal{U}_2A_1\right)}\cdots e^{T\left(\mathcal{U}_2A_r\right)} \cdot T\left(\mathcal{U}_2(A^{-1})\right).\label{A1A2 ver2}
    \end{align}
\end{thm}
\begin{proof}
As pointed out in \eqref{det-2},
        \begin{align*}
            \det(P_{h_1(n), h_2(n)}AP_{h_1(n), h_2(n)}) = \det (P_{h(n)}T(\mathcal{U}^{h_1(n)}A)P_{h(n)}),
        \end{align*}
and therefore it suffices to consider the latter. Let $(B_n)$ be the sequence defined by (\ref{Bn for SW thm}). Taking the determinant, we have
        \begin{align*}
      \lefteqn{       \det B_n}
      \\
       &= \det \left(P_{h(n)}T(\mathcal{U}^{h_1(n)}(A))P_{h(n)}\cdot e^{-P_{h(n)}T(\mathcal{U}^{h_1(n)}A_r)P_{h(n)}}\cdots
                    e^{-P_{h(n)}T(\mathcal{U}^{h_1(n)}A_1)P_{h(n)}}\right)\\
            &= \det \left(P_{h(n)}T(\mathcal{U}^{h_1(n)}(A))P_{h(n)}\right)\cdot e^{-\trace ( P_{h(n)}T(\mathcal{U}^{h_1(n)}A_r)P_{h(n)})}\cdots
                    e^{-\trace (P_{h(n)}T(\mathcal{U}^{h_1(n)}A_1)P_{h(n)})}\\
            &= \det \left(P_{h(n)}T(\mathcal{U}^{h_1(n)}(A))P_{h(n)}\right)\cdot e^{-\trace  (P_{h(n)}T(\mathcal{U}^{h_1(n)}(A_1+\cdots+A_r))P_{h(n)}) }\\
            &= \det \left(P_{h(n)}T(\mathcal{U}^{h_1(n)}(A))P_{h(n)}\right)\cdot e^{-\trace (P_{h_1(n), h_2(n)}(A_1+\cdots+A_r)P_{h_1(n), h_2(n)})}.
        \end{align*}
        In the last step we used \eqref{tr-2}.
By Proposition \ref{prop Bn}, $B_n$ is of the form (\ref{Bn simp for SW thm}), where
        \begin{align*}
            \det(P+K) &= \det \left(T(\mathcal{U}_1A)\cdot e^{-T(\mathcal{U}_1A_r)}\cdots e^{-T(\mathcal{U}_1A_1)}\right),\\
            \det(P+L) &= \det \left(T(\widetilde{\mathcal{U}_2A})\cdot e^{-T(\widetilde{\mathcal{U}_2A_r})}\cdots e^{-T(\widetilde{\mathcal{U}_2A_1})}\right).
        \end{align*}
We apply Lemma \ref{det iden} to obtain the limit \eqref{f:5.7}.
Using Proposition \ref{prop det iden} the second operator determinant can be rewritten as
        \begin{equation*}
            \det \left(e^{T\left(\mathcal{U}_2A_1\right)}\cdots e^{T\left(\mathcal{U}_2A_r\right)} \cdot T\left(\mathcal{U}_2(A^{-1})\right)\right),
        \end{equation*}
and both determinants together yield the well-defined constant $\det(\mathcal{B}_1)\det(\mathcal{B}_2)$.
\end{proof}

\bigskip

In what follows we will apply the previous theorem to the Banach algebra 
$\mathcal{R}=\mathcal{W}_{\alpha_1,\alpha_2}(\mathcal{A})$ with $\mathcal{A} = APW(\Z,\Xi,\beta)$, where $\beta$ is an admissible and compatible weight 
on the additive subgroup $\Xi$ of $\R/\Z$, and $\alpha_1,\alpha_2\ge0$, $\alpha_1+\alpha_2=1$.
Note that  Corollary \ref{c3.4}(i) implies that $\cR$ is a rigid, suitable, shift-invariant, and unital Banach algebra.
Moreover, the compatibility condition will facilitate the computation of the trace by Theorem \ref{thm trace eva}.

Let us recall the notation that is used in the following theorem. First, $M(a)$ stands for the mean of an almost periodic sequence $a$  as defined in \eqref{mean.Ma}, and $D(A)\in l^\infty(\Z)$ denotes the main diagonal of an operator $A\in \mathcal{L}(l^2(\Z))$.

Furthermore, Corollary \ref{c3.4}(ii) implies that if $h$ is fractal for $\Xi$ with associated $\tau\in\HomXT$,
then $h$ is fractal for $\cR$ with associated $\mathcal{U}^\tau$ given by
\begin{equation}\label{mapU2}
\cU^\tau:\cR\to\cR,\quad \sum_{n\in\Z} \sum_{\xi\in \Xi} a^{(n)}_\xi (e_\xi I)U_n \mapsto
\sum_{n\in\Z} \sum_{\xi\in \Xi} a^{(n)}_\xi \tau(\xi) (e_\xi I)U_n.
\end{equation}
In fact,  it is easy to see that $\mathcal{U}^\tau$ is well-defined for all $\tau\in \HomXT$ and is a unital isometric
Banach algebra isomorphism on $\cR$ (see also Proposition \ref{Theta is cont}(a) below).

Therefore, for any $\tau\in \rm{Hom}(\Xi, \mathbb{T})$ and any $A\in \mathcal{R}$ of the form
	\[A = e^{A_1}\cdots e^{A_r}\enskip\]
with $ A_1, \dots, A_r\in \mathcal{R}$, it is possible to define the constants
	\begin{align}\label{Theta_1}
		\Theta_{A,1}(\tau) &= \exp\left(\sum_{\xi\in\Xi,\xi\neq 0}a_\xi \frac{\tau(\xi)}{1-e^{2\pi i \xi}}\right) \cdot 
			\det\left(T(\mathcal{U}^{\tau}A)e^{-T(\mathcal{U}^{\tau}A_r)}\cdots e^{-T(\mathcal{U}^{\tau}A_1)}\right),\\ 
		\label{Theta_2}
		\Theta_{A,2}(\tau) &= \exp\left(\sum_{\xi\in\Xi,\xi\neq 0}a_\xi \frac{-\tau(\xi)}{1-e^{2\pi i \xi}}\right) \cdot 
			\det\left(e^{T(\mathcal{U}^{\tau}A_1)}\cdots e^{T(\mathcal{U}^{\tau}A_r)}T(\mathcal{U}^{\tau}A^{-1})\right),
	\end{align}
where
\begin{equation}\label{def.a.axi}
	a = \sum_{\xi\in\Xi} a_\xi e_\xi= D(A_1+\cdots + A_r)\in\cA.
\end{equation}
Observe that the operator determinants are well-defined by Proposition \ref{P plus trace class}.
Since $\beta$ is compatible on $\Xi$, the sums are finite as shown in Theorem \ref{thm trace eva}.
Note that the definition of the constants depends also on the choice of $A_1,\dots, A_r$, 
which is omitted in the notation.

\begin{thm}[Fractal version of the limit theorem]
\label{MainResult-1}
Let $\beta$ be an admissible and compatible weight on an additive subgroup $\Xi$ of $\R/\Z$. Let $\mathcal{R} = \mathcal{W}_{\alpha_1,\alpha_2}(\mathcal{A})$ with $\alpha_1, \alpha_2\ge 0$, $\alpha_1 +\alpha_2 = 1$, and $\mathcal{A} = APW(\Z,\Xi,\beta)$. 
Suppose that  $h_1$ and $h_2$ are  fractal sequences for $\Xi$ with associated $\tau_1$ and 
$\tau_2$ in $\mathrm{Hom}(\Xi, \mathbb{T})$, respectively,
such that $h(n) := h_2(n) - h_1(n)>0$ and $h(n)\to \infty$ as $n\to\infty$.

 If $A_1,\dots, A_r\in \mathcal{R}$ and
    $A = e^{A_1}\cdots e^{A_r}$,
then
\begin{equation}
      \lim_{n\to\infty} \frac{\det(P_{h_1(n), h_2(n)}AP_{h_1(n), h_2(n)})}{G^{h_2(n)- h_1(n)}} = \Theta_{A,1}(\tau_1) \Theta_{A,2}(\tau_2),    
\end{equation}
where
\begin{equation}\label{f5.12}
G=\exp(M(a)), \quad a = D(A_1+\cdots + A_r).
\end{equation}
\end{thm}
\begin{proof}
Because of Corollary \ref{c3.4}(i), Theorem \ref{MainResult-0} can be applied and we are left with the evaluation of the asymptotics for
    \begin{equation*}
        \text{trace}(P_{h_1(n), h_2(n)}(A_1+\cdots+A_r)P_{h_1(n), h_2(n)}) =  \sum_{k=h_1(n)}^{h_2(n)-1}a(k)
        =\text{trace}(P_{h_1(n), h_2(n)}(aI)P_{h_1(n), h_2(n)})
    \end{equation*}
with $a$ given by \eqref{f5.12}. 
Theorem \ref{thm trace eva} can be applied and yields
 \begin{equation*}
       \sum_{k=h_1(n)}^{h_2(n)-1}a(k) = \left(h_2(n) - h_1(n)\right)\cdot M(a) + F_a(\tau_1) - F_a(\tau_2)+ o(1),\qquad n\to\infty,
\end{equation*}
where $F_a(\tau)$ is given by (\ref{C-tau}). 
Combining the exponentials of the constants $F_a(\tau_k)$ with the operator determinants $\det(\mathcal{B}_k)$ gives the constants $\Theta_{A,k}(\tau_k)$.
\end{proof}


\section{A uniform version of the limit theorem}
\label{sec:6}

In this section let us assume that $\mathcal{R} = \mathcal{W}_{\alpha_1, \alpha_2}(APW(\Z, \Xi, \beta))$ 
with the same assumptions as before: $\alpha_1,\alpha_2\ge 0$, $\alpha_1+\alpha_2=1$, $\Xi$ is an additive subgroup of
$\R/\Z$, and $\beta$ is an admissible and compatible weight on $\Xi$.

In this setting the constants $\Theta_{A,1}(\tau)$ and $\Theta_{A,2}(\tau)$
are well-defined for all $\tau\in\HomXT$. For fixed $A=e^{A_1}\cdots e^{A_r}$ with $A_1,\dots A_r\in\cR$ 
these quantities can be considered as functions on $\HomXT$.
The goal of the following two propositions is to show that these functions are continuous on  $\HomXT$.
After this we will derive the uniform version of the limit theorem.

Recall that $\mathrm{Hom}(\Xi, \mathbb{T})$ is a compact topological space with the topology determined by the local bases of the form
\begin{equation}\label{local bases}
U_{\xi_1,\dots,\xi_N;\varepsilon}[\tau]=\Big\{\, \tau'\in\HomXT \,:\, |\tau'(\xi_k)-\tau(\xi_k)|<\varepsilon \mbox{ for all } 1\le k\le N\,\Big\},
\end{equation}
where $\varepsilon>0$, $N\in\mathbb{N}$, $\xi_1,\dots,\xi_N\in\Xi$.

\begin{prop}\label{C is cont}
Let $\beta$ be an admissible weight on an additive subgroup $\Xi$ of $\R/\Z$. For each fixed $a=\sum\limits_{\xi\in\Xi} a_\xi e_\xi \in \mathcal{A}= APW(\Z, \Xi, \beta)$,  
\begin{enumerate}[label=(\alph*)]
\item the function $\psi_a:\HomXT\to\cA$ defined by
	\begin{equation}
		\psi_a(\tau) = U^\tau a := \sum_{\xi\in\Xi} a_\xi \tau(\xi) e_\xi
			\end{equation}
is well-defined and continuous;
\item 
the function $F_a: \HomXT\to \C$ defined by
$$
F_a(\tau)=\sum_{\xi\in\Xi,\xi\neq0} a_\xi\,\frac{\tau(\xi)}{1-e^{2\pi i \xi}}
$$
is well-defined and continuous provided the weight $\beta$ is compatible.
\end{enumerate}
\end{prop}
\begin{proof}
(a):\enskip 
Since the spectrum of $a$ is at most countable and $\sum\limits_{\xi\in \Xi}|a_\xi|\beta(\xi)<\infty$,  for any 
given $\epsilon>0$, there is a finite subset $S = \{\xi_1, \dots, \xi_n\}$ of $\Xi$, such that
		\[\sum_{\xi\in \Xi\setminus S} |a_\xi|\cdot \beta(\xi) <\frac{\epsilon}{4}.\]
For each fixed $\tau\in \mathrm{Hom}(\Xi, \mathbb{T})$, consider $U = U_{\xi_1,\cdots, \xi_n,\epsilon_0}[\tau]$ defined by (\ref{local bases})  with $\epsilon_0 = \frac{\epsilon}{2\|a\|_{\mathcal{A}}}$, and note that 
		\begin{align*}
			\|U^\tau a - U^\tp a\|_\mathcal{A} &= \sum_{\xi\in S} |a_\xi|\cdot|\tau(\xi) - \tp(\xi)|\beta(\xi)
												+ \sum_{\xi\in \Xi\setminus S } |a_\xi|\cdot|\tau(\xi) - \tp(\xi)|\beta(\xi)\\
						&< \epsilon_0 \sum_{\xi\in S}|a_\xi|\beta(\xi) +  2\sum_{\xi\in \Xi\setminus S}|a_\xi|\beta(\xi)\\
						&<\frac{\epsilon}{2}+\frac{\epsilon}{2} = \epsilon
		\end{align*}
for any $\tp\in U$. It implies that $\psi_a$ is a continuous function on $\mathrm{Hom}(\Xi, \mathbb{T})$ for each fixed $a$.

(b):\enskip	
Similarly, for any given $\epsilon>0$, there exists a finite subset $S = \{\xi_1, \cdots, \xi_n\}$ of $\Xi$, such that
		\[\sum_{\xi\in \Xi\setminus S} |a_\xi|\cdot \beta(\xi) <C_\beta\epsilon,\]
where $C_\beta$ is the constant given in (\ref{compat.cond}). For each fixed $\tau\in \mathrm{Hom}(\Xi, \mathbb{T})$, consider the open neighborhood $U = U_{\xi_1,\cdots, \xi_n,\epsilon_0}[\tau]$ with $\epsilon_0 = \frac{C_\beta \epsilon}{\|a\|_{\mathcal{A}}}$. Then, for any $\tp\in U$, using estimate \eqref{estimate-beta}
		\begin{align*}
			|F_a(\tau) - F_a(\tp)| &\le \sum_{\xi\in S, \xi\neq 0} |a_\xi| \frac{|\tau(\xi) - \tp(\xi)|}{|1-e^{2\pi i\xi}|} + \sum_{\xi \in \Xi \setminus S,\xi\neq 0}  |a_\xi| \frac{|\tau(\xi) - \tp(\xi)|}{|1-e^{2\pi i\xi}|}\\
								&\le \frac{1}{4C_\beta}\sum_{\xi\in S, \xi\neq 0} |a_\xi|\beta(\xi)\cdot |\tau(\xi)-\tp(\xi)| + \frac {1}{2C_\beta} \sum_{\xi\in \Xi \setminus S} |a_\xi| \beta(\xi)\\
								&< \frac{\epsilon}{2}  +  \frac{\epsilon}{2}  = \epsilon.
		\end{align*}
Since the choice of $\epsilon$ and $\tau$ is arbitrary,  $F_a$ is a continuous function on $\mathrm{Hom}(\Xi, \mathbb{T})$.
\end{proof}

\begin{prop}\label{Theta is cont}
 Let $\beta$ be an admissible weight on an additive subgroup $\Xi$ of $\R/\Z$. Let $\mathcal{R} = \mathcal{W}_{\alpha_1,\alpha_2}(\mathcal{A})$ with $\alpha_1, \alpha_2\ge 0$, $\alpha_1 +\alpha_2 = 1$, and $\mathcal{A} = APW(\Z,\Xi,\beta)$. 
\begin{enumerate}[label=(\alph*)]
\item 
For each fixed $A\in \mathcal{R}$, the function $\Psi_A(\tau)$ defined by
\begin{equation}
\Psi_A: \mathrm{Hom}(\Xi, \mathbb{T})\to \mathcal{R}, \enskip \tau \mapsto \mathcal{U}^\tau A
\end{equation}
is well-defined and continuous. 

\item 
For each fixed $A = e^{A_1}\cdots e^{A_r}$ with  $A_1, \dots, A_r\in \mathcal{R}$, the functions 
\begin{align*}
{\Delta}_{A, 1}(\tau) 
&= 
\det\left(T(\mathcal{U}^{\tau}A)e^{-T(\mathcal{U}^{\tau}A_r)}\cdots e^{-T(\mathcal{U}^{\tau}A_1)}\right), 
\\[1ex]
{\Delta}_{A, 2}(\tau) 
&= 
\det\left(e^{T(\mathcal{U}^{\tau}A_1)}\cdots e^{T(\mathcal{U}^{\tau}A_r)}T(\mathcal{U}^{\tau}A^{-1})\right)
\end{align*}
are well-defined and continuous on $\mathrm{Hom}(\Xi, \mathbb{T})$. 

\item
For each fixed $A = e^{A_1}\cdots e^{A_r}$ with  $A_1, \dots, A_r\in \mathcal{R}$, the functions 
 $\Theta_{A, 1}(\tau)$ and $\Theta_{A, 2}(\tau)$ given by (\ref{Theta_1}) and (\ref{Theta_2}) 
 are well-defined and continuous on $\mathrm{Hom}(\Xi, \mathbb{T})$ provided that the weight $\beta$ is compatible. 
\end{enumerate}
\end{prop}
\begin{proof}
(a):\enskip Recall that
	\[\|A\|_{\mathcal{R}} = \sum_{k\in \Z}\alpha(k)\|a^{(k)}\|_\mathcal{A} = \sum_{k\in \Z}\alpha(k)\sum_{\xi\in \Xi}|a^{(k)}_\xi|\beta(\xi),\]
where
	\[
	A =: \sum_{k\in \Z}(a^{(k)}I)U_k
	\quad\mbox{and}\quad
 a^{(k)} =: \sum_{\xi\in \Xi}a^{(k)}_\xi e_\xi.
 \]
For any given $\epsilon>0$, there exist $K\in \mathbb{N}$ and a finite subset $S = \{\xi_1, \dots, \xi_n\}$ of $\Xi$, such that
	\begin{equation*}
		\sum_{|k|\ge K} \alpha(k)\|a^{(k)}\|_\mathcal{A} < \frac{\epsilon}{6},
	\end{equation*}
and
	\begin{equation*}
		\sum_{\xi\in \Xi\setminus S}|a^{(k)}_\xi|\beta(\xi) < \frac{\epsilon}{12KM_K}
	\end{equation*}
whenever $|k|< K$, where $M_K = \max\{\alpha(k): |k|<K\}$. For a fixed $\tau\in \rm{Hom}(\Xi, \mathbb{T})$, consider the open neighborhood $U = U_{\xi_1,\cdots, \xi_n,\epsilon_0}[\tau]$ with $\epsilon_0 = \frac{\epsilon}{3\|A\|_{\mathcal{R}}}$. We have 
	\begin{align*}
		\|\mathcal{U}^\tau A - \mathcal{U}^\tp A\|_\mathcal{A} 
						&= \sum_{|k|<K} \alpha(k) \|U^\tau a^{(k)} - U^\tp a^{(k)}\|_\mathcal{A} + \sum_{|k|\ge K} \alpha(k)\|U^\tau a^{(k)} - U^\tp a^{(k)}\|_\mathcal{A}\\
						&< \sum_{|k|<K}\alpha(k) \left( \sum_{\xi\in S}|a^{(k)}_\xi|\cdot|\tau(\xi) - \tp(\xi)|\beta(\xi)  + 2\sum_{\xi\in \Xi\setminus S}|a^{(k)}_\xi|\beta(\xi)\right) + 2\cdot \frac{\epsilon}{6}\\
						&< \sum_{|k|<K}\alpha(k) \epsilon_0 \sum_{\xi\in S}|a^{(k)}_\xi|\beta(\xi) + \frac{\epsilon}{6KM}\sum_{|k|<K}\alpha(k) + \frac{\epsilon}{3}\\
						&<\frac{\epsilon}{3}+\frac{\epsilon}{3}+\frac{\epsilon}{3} = \epsilon
	\end{align*}
for any $\tp\in U$. Since the choices of $\epsilon$ and $\tau$ are arbitrary, $\Psi_A$ is a continuous function on $\rm{Hom}(\Xi, \mathbb{T})$ for each fixed $A$.

(b):\enskip It suffices to show that the two operators 
	\begin{align*}
		\mathcal{B}_1(\tau) &:=  T(\mathcal{U}^{\tau}A)e^{-T(\mathcal{U}^{\tau}A_r)}\cdots e^{-T(\mathcal{U}^{\tau}A_1)},\\
		\mathcal{B}_2(\tau) &:=  e^{T(\mathcal{U}^{\tau}A_1)}\cdots e^{T(\mathcal{U}^{\tau}A_r)}T(\mathcal{U}^{\tau}A^{-1})
	\end{align*}
are both continuous functions from $\mathrm{Hom}(\Xi, \mathbb{T})$ to $\mathcal{O}(\mathcal{R})$. Note that, 
	\[\|T(\mathcal{U}^\tau A) - T(\mathcal{U}^\tp A)\|_{\mathcal{O}(\mathcal{R})}= \|\mathcal{U}^\tau A - \mathcal{U}^\tp A\|_\mathcal{R},\]
and thus
	\begin{equation}
		\chi_A: \mathrm{Hom}(\Xi, \mathbb{T})\to \mathcal{O}(\mathcal{R}), \enskip \tau \mapsto T(U^\tau A)
	\end{equation}
is continuous for each fixed $A\in \mathcal{R}$ by part (a).  

Next, note that the exponential function is continuous in any Banach algebra, and thus
	\[\eta_A: \mathrm{Hom}(\Xi, \mathbb{T})\to \mathcal{O}(\mathcal{R}), \enskip \tau \mapsto e^{-T(U^\tau A)}\]
is continuous from $\rm{Hom}(\Xi, \mathbb{T})$ to $\mathcal{O}(\mathcal{R})$ for each fixed $A\in \mathcal{R}$. Finally, for fixed $A_1, \dots A_r\in \mathcal{R}$ and $A = e^{A_1}\cdots e^{A_r}$, 
we have that
	\begin{align*}
		\mathcal{B}_1(\tau) &= \chi_A(\tau)\eta_{A_r}(\tau)\cdots \eta_{A_1}(\tau),\\
		\mathcal{B}_2(\tau) &= \eta_{-A_1}(\tau) \cdots \eta_{-A_r}(\tau)\chi_{A^{-1}}(\tau)
	\end{align*}	
are  continuous functions from $\rm{Hom}(\Xi, \mathbb{T})$ to $\mathcal{O}(\mathcal{R})$, and in addition, $\mathcal{B}_1(\tau) - P$ and $\mathcal{B}_2(\tau)-P$ are both trace class by Proposition \ref{P plus trace class}. It follows that 
${\Delta}_{A, 1}(\tau)$ and ${\Delta}_{A,2}(\tau)$ are well-defined continuous on $\rm{Hom}(\Xi, \mathbb{T})$.

Part (c) follows directly from part (b) above and from Proposition \ref{C is cont}(b).
\end{proof}

We also need the following simple lemma.
Therein the additional assumption on $\Xi$ being at most countable is imposed.

\begin{lem}\label{frac subsequence}
	Let $\Xi$ be an at most countable additive subgroup of $\R/\Z$. For any sequence $n=\{n(k)\}_{k=1}^\infty$ of integers
	there exists a subsequence $h=\{h(k)\}_{k=1}^\iy$ of $n$ which is fractal for $\Xi$. 
\end{lem}
\begin{proof}
The proof is based on a standard diagonal argument.
	Since $\Xi$ is countable we can assume that $\Xi = \{\xi_t: t\in \mathbb{N}\}$. 
Starting with $n_0:=n$ one can recursively construct a collection of integer sequences $n_t=\{n_t(k)\}_{k=1}^\infty$, $t\in\mathbb{N}$, such that $n_t$ is a subsequence of $n_{t-1}$  and such that $\{e^{2\pi i n_t(k) \xi_t}\}_{k=1}^\iy$
converges for every fixed $t\in\mathbb{N}$.
Indeed, suppose we are given $n_{t-1}$. Then we can consider $\{e^{2\pi i n_{t-1}(k) \xi_t}\}_{k=1}^\infty$ and select a 
convergent subsequence $\{e^{2\pi i n_t(k) \xi_t}\}_{k=1}^\infty$, thereby defining a subsequence $n_t$ of $n_{t-1}$. Moreover, we can keep
the first $t-1$ terms  of the sequence $n_{t-1}$ unchanged when passing to $n_t$, i.e., $n_t(k)=n_{t-1}(k)$, $k=1,\dots, t-1$.

Having defined sequences $n_t=\{n_t(k)\}_{k=1}^\infty$ for all $t\in\mathbb{N}$, we define 
$$
h(k)=:n_k(k).
$$
It is straightforward to verify that $h=\{h(k)\}_{k=1}^\infty$	is a subsequence of each sequence $n_t$, $t=0,1,\dots$, in particular of the sequence $n=n_0$. Now $h$ being a subsequence of $n_t$ implies that 
$\{e^{2\pi i h(k) \xi_t}\}_{k=1}^\infty$ converges for every fixed $t\in\mathbb{N}$. 
But this implies that $h=\{h(k)\}_{k=1}^\infty$ is fractal for $\Xi$ by the definition of fractality.
\end{proof}

For each $n\in \Z$ we define a corresponding $\tau_n\in \mathrm{Hom}(\Xi, \T)$ by
	\begin{equation}
		\tau_n(\xi):=e^{2\pi i n\xi}, \quad  \xi \in \Xi.
	\end{equation}
As a consequence of Theorem \ref{MainResult-1} we obtain the following version of the limit theorem.

\begin{thm}[Uniform version of the limit theorem]\label{MainResult-2}
Let $\beta$ be an admissible and compatible weight on an at most countable additive subgroup $\Xi$ of $\R/\Z$. 
Let $\mathcal{R} = \mathcal{W}_{\alpha_1,\alpha_2}(\mathcal{A})$ 
with $\alpha_1, \alpha_2\ge 0$, $\alpha_1 +\alpha_2 = 1$, and $\mathcal{A} = APW(\Z,\Xi,\beta)$.

If $A_1,\dots, A_r\in \mathcal{R}$ and $A = e^{A_1}\cdots e^{A_r}$,
then
\begin{equation}\label{limit6.8a}
		\lim_{n_2 - n_1\to  \infty} \left(	  \frac{\det(P_{n_1, n_2}AP_{n_1, n_2})}{G^{n_2 - n_1}} - \Theta_{A, 1}(\tau_{n_1})\Theta_{A, 2}(\tau_{n_2})		\right) = 0,
\end{equation}
where	
	\begin{equation*}
        	G=\exp(M(a)), \quad a = D(A_1+\cdots + A_r),
        \end{equation*}
and with $\Theta_{A, 1}$ and $\Theta_{A, 2}$ given by \eqref{Theta_1} and \eqref{Theta_2}.
\end{thm}

\begin{proof}
For each pair of integers $(n_1, n_2)$ such that $n_2 - n_1>0$ define the quantity
$$
		F[n_1, n_2]:=\frac{\det(P_{n_1, n_2}AP_{n_1, n_2})}{G^{n_2 - n_1}} - \Theta_{A, 1}(\tau_{n_1})\Theta_{A, 2}(\tau_{n_2}).
$$
We will prove the theorem by contradiction. Assume that \eqref{limit6.8a} does not hold.
Then there exists an $\epsilon >0$ such that for each  $k\in \mathbb{N}$ there exists two integers $n_1(k)$ and $n_2(k)$ 
with $n_2(k) - n_1(k) > k$ such that
\[
|F[n_1(k), n_2(k)]| \ge \epsilon.
\]
Thus we obtain two integer sequences $\{n_1(k)\}_{k=1}^\infty$ and  $\{n_2(k)\}_{k=1}^\infty$ such that $n_2(k) - n_1(k)\to  \infty$ as $k\to \infty$. By applying Lemma \ref{frac subsequence} twice, there exists a strictly increasing sequence 
$\{k_j\}_{j=1}^\iy$ of positive integers such that $h_1=\{h_1(j)\}_{j=1}^\infty:=\{n_1(k_j)\}_{j=1}^\infty$ and $h_2=\{h_2(j)\}_{j=1}^\infty:=\{n_2(k_j)\}_{j=1}^\infty$ are both fractal sequences for $\Xi$.  
Assume that these fractal sequences $h_1$ and $h_2$ have associated $\tau_{h_1},\tau_{h_2}\in \HomXT$.
Then  $\tau_{h_i(j)}$ converges to $\tau_{h_i}$ in the natural topology of $\mathrm{Hom}(\Xi, \T)$ as $j\to\infty$
(for $i=1,2$), and therefore
$$
\Theta_{A,1}(\tau_{h_1(j)})\Theta_{A,2}(\tau_{h_2(j)})\to \Theta_{A,1}(\tau_{h_1})\Theta_{A,2}(\tau_{h_2})
$$
as $j\to\infty$ by Proposition \ref{Theta is cont}. On the other hand,
$$
   \frac{\det(P_{h_1(j), h_2(j)}AP_{h_1(j), h_2(j)})}{G^{h_2(j) - h_1(j)}} \to \Theta_{A, 1}(\tau_{h_1})\Theta_{A, 2}(\tau_{h_2})
$$
as $j\to \infty$ by Theorem \ref{MainResult-1}, which contradicts the fact that 
$|F[h_1(j), h_2(j)]| \ge \epsilon$ for each $j\in \mathbb{N}$ since $h_1$ and $h_2$ are subsequences of 
$\{n_1(k)\}_{k=1}^\infty$ and  $\{n_2(k)\}_{k=1}^\infty$, respectively. 	
\end{proof}

Note that the additional assumption that 
$\Xi$ is at most countable is not a serious restriction. Indeed, for a given operator $A$ with almost periodic diagonals, one 
can take for $\Xi$ the additive subgroup of $\R/\Z$ generated by the union of the Fourier spectra of
all the diagonals $D_k(A)\in AP(\Z)$ of $A$. Recall that the Fourier spectrum of any sequence in $AP(\Z)$ is at most countable.

\section{Special cases and additional remarks}
\label{sec:7}

If one wants to verify whether the limit theorems (Theorem \ref{MainResult-1} or Theorem \ref{MainResult-2}) can be applied
to a concrete operator $A\in\mathcal{OAP}\subseteq \mathcal{L}(l^2(\Z))$  one faces the following problems:

\begin{itemize}  
\item[(i)]
Does there exist an admissible and compatible weight $\beta$ on $\Xi$,
where $\Xi$ is a subgroup of $\R/\Z$ containing the  Fourier spectra of all the diagonals 
$D_k(A)\in AP(\Z)$ of $A$ ?
\item[(ii)] 
Does the operator $A$ belong to $\cR=\mathcal{W}_{\alpha_1,\alpha_2}(APW(\Z,\Xi,\beta))$ ?
\item[(iii)]
Moreover, can $A$ be written as a product of exponentials of operators in $\cR$ ?
\end{itemize}
As we will see below, the first problem is related to diophantine approximation, and the last problem 
naturally leads to the question whether $\cR$ is inverse closed in $\mathcal{L}(l^2(\Z))$.

The questions are also related to each other in the following sense. 
On the one hand, the compatibility condition requires the weight 
$\beta$ to grow sufficiently fast (depending on $\Xi$). A fast growing weight, on the other hand, severely restricts 
the class $\cR$. In addition, it may prevent the inverse closedness of $\cR$ in $\mathcal{L}(l^2(\Z))$. For this reason, 
it would be desirable to consider (admissible and compatible) weights $\beta$
satisfying the {\em Gelfand-Raikov-Shilow condition} (or, {\em GRS-condition}), namely
\begin{equation}
\lim_{k\to \infty} \beta(k\xi)^{1/k}=1
\quad\mbox{ for all }\xi\in \Xi.
\end{equation}
(Therein $k$ is a positive integer and $k\xi\in \Xi$ is well-defined.)
In fact, this conditions is necessary for the inverse closedness of  $\cR$ in $\mathcal{L}(l^2(\Z))$.

\medskip
Consequently,  we may pose the following question: For which (countable) subgroups $\Xi$ of $\R/\Z$ do there exist
admissible and compatible weights $\beta$ satisfying in addition the GRS-condition?

This question is not easy to answer in general, and we will discuss it to some extent for finitely
generated groups $\Xi$. However, let us first present a simple positive example (mentioned already in \cite[Ex.~2.9]{ERS})
as well as a (not so simple) counter-example.

\begin{example}
Let $\Xi=\Q/\Z$. This is a countable subgroup of $\R/\Z$ which is not finitely generated.
Yet, we can define an admissible and compatible weight on $\Xi$ simply  by 
$\beta(\xi)=q$ whenever $\xi=[p/q]$ with $p\in\Z$, $q\in\mathbb{N}$ being coprime.
This weight also satisfies the GRS-condition.
\end{example}

Here and in what follows, we will notationally distinguish between an equivalence class $[x]\in\R/\Z$ and
its representative $x\in\R$. Furthermore, for a given subset $S\subseteq \R/\Z$, let $\gr S$ stand for the subgroup 
of $\R/\Z$ generated by the set $S$. In other words, $\gr S$ consists of all 
finite integer linear combinations of elements from $S$.

\begin{example}
Let $\Xi_\xi=\gr\{[\xi]\}$ denote the subgroup generated by an irrational number $\xi$.
In Section 6 of \cite{ERS} (see Example 6.7 in particular), the following class of examples has been exhibited.

Let $1<b<c$ and $\alpha$ be real numbers such that $0<\alpha<1-\frac{\log b}{\log c} <1$.
Then one can construct a Liouville number $\xi$ and a (stricly increasing) sequence $h=\{h(n)\}_{n=1}^\infty$
of integers such that
\begin{itemize}
\item[(i)]
the sequence $h$ is distinguished for $\Xi_\xi$ in sense of \cite{ERS} (hence fractal for $\Xi_\xi$ in our sense), and
\item[(ii)]
$$
\sum_{k=0}^{h(n)-1} a(k) = h(n)^\alpha(1+o(1)),\qquad n\to \infty,
$$
where $a=\sum_{k=1}^\infty b^{-k} e_{k \xi}\in AP(\Z)$.
\end{itemize}
This class of examples has the following consequences. Suppose there exists an admissible and compatible weight $\beta$ on $\Xi_\xi$
satisfying in addition 
$$
\lim_{k\to \infty} \beta(k \xi)^{1/k}<b,
$$
(the latter being the case if $\beta$ satisfies the GRS-condition). Then $a\in APW(\Z,\Xi,\beta)$, and
 it would be possible to apply Theorem \ref{thm trace eva} with $h_1(n)=0$ and $h_2(n)=h(n)$
and obtain the asymptotics \eqref{trace-conv}. But this contradicts the asymptotics given in (ii).

Therefore we can conclude that for the Liouville numbers $\xi$ constructed above, there exists no weight $\beta$ on $\Xi_\xi$
which is admissible, compatible, and satisfies the GRS-condition.

It is an open question whether there exist Liouville numbers $\xi$ for which there exists no admissible and 
compatible weight $\beta$ on $\Xi_\xi$ (regardless of the GRS-condition).
\end{example}

\subsection{The structure of finitely generated subgroups $\boldsymbol{\Xi}$ of $\boldsymbol{\R/\Z}$}
\label{sec:7.1}

Let us now consider finitely generated subgroups $\Xi$ of  $\R/\Z$. Such subgroups are necessarily countable.
The following two results characterize the structure of such subgroups.
Therein, let $\Z_N=\Z/(N\Z)$ denote the set of congruence classes modulo $N$.
The proof of the first proposition is straightforward. The second one could certainly be derived from the general structure
theorem for finitely generated abelian groups, but we provide a proof which is constructive.

\begin{prop}\label{fin gen Xi}
For $n\ge0$, let $\xi_1, \dots, \xi_n$ be real numbers such that $\{\xi_1, \dots, \xi_n,1\}$ is linearly independent over $\mathbb{Q}$, and let $N\in\mathbb{N}$.
Then the map
\begin{equation}\label{group-iso}
(\alpha_1,\dots,\alpha_n,\alpha_{n+1})\in \Z^n\times \Z_N
\mapsto
	\left[\alpha_1\xi_1 +\cdots \alpha_n\xi_n + \alpha_{n+1}\frac{1}{N}\right] \in\R/\Z
\end{equation}
is a well-defined group isomorphism between the additive group  $\Z^n\times \Z_N$ and the 
finitely generated subgroup $\Xi = \gr\left\{ [\xi_1], \dots, [\xi_n], [\frac{1}{N}] \right\}$ of $\R/\Z$.
\end{prop}

\begin{prop}\label{char Xi}
If  $\Xi$ is a finitely generated subgroup of $\R/\Z$, then there exists $n\ge0$ and $N\ge 1$ such that 
$\Xi$ is group-isomorphic to $\Z^n\times \Z_N$ via an isomorphism of the form \eqref{group-iso}.
\end{prop}
\begin{proof}
The group $\Xi$ being finitely generated means that there exist generators $[\xi_1],\dots,[\xi_m]\in\R/\Z$ such that 
$$
\Xi=\gr \left\{ [\xi_1], \dots, [\xi_m]   \right\}:=\left\{\sum_{k=1}^m \alpha_k[\xi_k]\,:\, \alpha_k\in\Z\right\}.
$$
The numbers $\xi_k$ therein are either irrational or rational. Among all tuples $([\xi_1],\dots,[\xi_m])$ of elements in $\R/\Z$ generating the group $\Xi$ 
consider one in which the number of irrational generators is as small as possible, say equal to $n$. Therefore, without loss of generality, we can assume that 
\begin{equation}\label{gen-1}
\Xi=\gr \left\{  [\xi_1], \dots, [\xi_n], [\xi_{n+1}],\dots,[\xi_{n+k}]   \right\}
\end{equation}
where $\xi_1\dots,\xi_n\in\R\setminus\Q$, $\xi_{n+1},\dots,\xi_{n+k}\in\Q$, and $n,k\ge0$.

Concerning the number $k$ of rational generators we made no assumption at this point. However, it is possible to modify the rational generators in \eqref{gen-1} and
 keeping the irrational generators unchanged such that the resulting representation has precisely one rational generator and in addition that this rational generator equals  $[\frac{1}{N}]$ for some integer $N\ge1$.
In other words, we can turn \eqref{gen-1} into
\begin{equation}\label{gen-2}
\Xi=\gr \left\{  [\xi_1], \dots, [\xi_n], [{\textstyle \frac{1}{N}}]   \right\}.
\end{equation}
The argument is as follows. First of all, if $k=0$, we can add the `dummy'  element $[1]=[0]$.
If $k\ge1$, then each rational generator $[\frac{p}{q}]$ in which $p\in\Z$ and $q\ge1$ are coprime can be replaced by $[\frac{1}{q}]$.
This is because $\gr\{[\frac{p}{q}]\}=\gr\{[\frac{1}{q}]\}$. If we have more than one rational generator ($k\ge2$) then we can replace
two generators $[\frac{1}{q_1}]$ and $[\frac{1}{q_2}]$ by a single generator $[\frac{1}{q}]$ with $q$ being the least common multiple of $q_1$ and $q_2$.
Here the reason is that $\gr\{[\frac{1}{q_1}],[\frac{1}{q_2}]\}=\gr\{[\frac{1}{q}]\}$.

We can now assume that $\Xi$ is given by \eqref{gen-2} and that the number $n$ of irrational generators cannot be reduced.
We claim that $\{\xi_1,\dots,\xi_n,1\}$ is rationally independent over $\Q$.
Suppose this is not the case.  Then $n\ge1$ and there are $a_1,\dots,a_n\in\Z$  such that 
$$
a_1\xi_1+\dots +a_n\xi_n=\frac{p}{q}\in\Q
$$
while not all $a_1,\dots,a_n$ are zero. In fact, one can assume that $\gcd(a_1,\dots,a_n)=1$. 
It is well-known and straightforward to show (e.g., by induction on $n$) that 
there exists a unimodular $n\times n$ matrix $M$ with entries in $\Z$ whose last row is 
$(a_1,\dots, a_{n})$.  Define the elements $[\xi_1'], \dots,[\xi_n']$ by
\begin{equation*}
		\begin{bmatrix}
			[\xi_1'] \\ \vdots \\ [\xi_{n}']
		\end{bmatrix} 
		:= M
		\begin{bmatrix}
			[\xi_1] \\ \vdots  \\ [\xi_{n}]
		\end{bmatrix}
	\end{equation*}
and note that $[\xi_n']=[\frac{p}{q}]$. Because $M^{-1}$ has entries in $\Z$ as well, it is easily seen that
$$
\gr   \left\{  [\xi_1], \dots, [\xi_n]  \right\}
=
\gr  \left\{  [\xi_1'], \dots, [\xi_n']  \right\}.
$$
Therefore,
$$
\Xi=\gr  \left\{  [\xi_1'], \dots, [\xi_{n-1}'], [{\textstyle\frac{p}{q}}], [{\textstyle\frac{1}{N}}]  \right\}.
$$
But this means that the number of irrational generators of $\Xi$ can be reduced, contrary to our assumption.

We therefore conclude that $\{\xi_1,\dots,\xi_n,1\}$ is linearly independent over $\Q$, which makes it possible to apply Proposition \ref{fin gen Xi}.
Since $\Xi$ is given by \eqref{gen-2}, it follows that $\Xi$ is group-isomorphic to $\Z^n\times\Z_N$ via the map \eqref{group-iso}.
\end{proof}

\medskip
After having seen that every finitely generated subgroup of $\R/\Z$ is group-isomorphic to
$\Z^n\times \Z_N$, the question now is for which of them does there exist an admissible and compatible weight
satisfying the GRS-condition. This depends, of course, on the irrational generators
$\xi_1,\dots,\xi_n$ featured in the description of the group $\Xi$.

\medskip
Before going into this, let us mention the trivial case of $n=0$, i.e., $\Xi=\gr\{[\frac{1}{N}]\}\cong \Z_N$. Since this group is finite the choice of the weight $\beta$ is irrelevant (one can take $\beta\equiv 1$). The corresponding Banach algebra 
$APW(\Z,\Xi,\beta)$ consists of all sequences in $l^\iy(\Z)$ that have period $N$. 
Even though this is the trivial case, Theorem \ref{MainResult-1} or Theorem \ref{MainResult-2} implies a generalization of the classical strong Szeg\"o-Widom limit theorem (see Subsection \ref{sec:7.3} below).

\bigskip

Let us now consider the case $n\ge1$. We denote by $\mathcal{S}_n \subseteq (\R/\Z)^n$ the set of all $(\xi_1, \dots, \xi_n)$ for which there exist $\omega>0$ and $C>0$ such that
\begin{equation}\label{f7.5}
        \|\alpha_1\xi_1+\cdots+\alpha_n\xi_n\|_{\R/\Z} \ge C \left(\max_{1\le i\le n} |\alpha_i|\right)^{-\omega}
\end{equation}
for all $(\alpha_1,\dots,\alpha_n)\in\Z^n\setminus\{0\}$.

As the following result shows, absence from $\mathcal{S}_n$ is rather exceptional. For a proof see, e.g., 
Section 3.5.3 ($n=1$) and Section 4.3.2 ($n\ge 2$) of \cite{BerDod}. 
Note that  the set $\mathcal{S}_1$ corresponds to irrational numbers which are not Liouville numbers.

\begin{thm}
The complement of $\mathcal{S}_n$ in $(\R/\Z)^n$ has Hausdorff dimension $n-1$, hence Lebesgue measure zero in $(\R/\Z)^n$.
\end{thm}

It is easy to see that condition \eqref{f7.5} implies that $\{\xi_1,\dots,\xi_n,1\}$ is linearly independent over $\Q$.
Hence for each $N\ge1$, the group
\begin{equation}\label{gr.xi}
\Xi=
\gr \left\{  [\xi_1], \dots, [\xi_n], [{\textstyle \frac{1}{N}}]   \right\}
\end{equation}
can be identified with $\Z^n\times \Z_N$ as described in Proposition \ref{fin gen Xi}.

\begin{prop}\label{p7.6}
For $n\ge1$, let $(\xi_1,\dots,\xi_n)\in \mathcal{S}_n$, and let $\omega$ be the constant in \eqref{f7.5}. 
Then for each $N\ge1$  one can define an admissible and compatible weight on
    \[\Xi = \left\{ \xi=
    \Big[ \alpha_1\xi_1+\cdots+\alpha_n\xi_n+\alpha_{n+1}\frac{1}{N}\Big]\,:\, \alpha = (\alpha_1, \dots, \alpha_n,\alpha_{n+1})\in \Z^n\times \Z_N \right\},
    \]
satisfying the GRS-condition, by
\begin{equation}\label{weight.finite.group}
\beta(\xi):= (1+|\alpha|)^\omega, \qquad |\alpha| = \max_{1\le i \le n}|\alpha_i|.
\end{equation}
\end{prop}
\begin{proof}
Obviously, the weight is admissible and satisfies the GRS-condition. 
Note that $\|\frac{x}{N}\|_{\R/\Z}\ge \frac{1}{N}\|x\|_{\R/\Z}$.   Therefore, whenever $|\alpha|\neq0$,
\begin{align*}
\|\xi\|  &=
\left\|\alpha_1\xi_1+\dots+\alpha_n\xi_n+\frac{\alpha_{n+1}}{N}\right\|
\ge \frac{1}{N}\left\| N\alpha_1\xi_1+\dots+N \alpha_n\xi_n+\alpha_{n+1}\right\|
\\
&
= \frac{1}{N}\left\| N\alpha_1\xi_1+\dots+N \alpha_n\xi_n\right\| 
\ge \frac{C}{N^{1+\omega}} (1+ |\alpha|)^{-\omega}
=\frac{C}{N^{1+\omega}\beta(\xi)}.
\end{align*}
Taking also into account the trivial case of $|\alpha|=0$, $\alpha_{n+1}\neq0$, this implies that $\beta$ is compatible
by the definition \eqref{compat.cond}.
\end{proof}

\bigskip
In order to practically apply the previous proposition one would have to know when $(\xi_1,\dots,\xi_n)\in\mathcal{S}_n$. 
We confine ourselves to briefly state two examples, both of which rely on deep results about diophantine approximation
(see, e.g., \cite{Cassels, MR0203515}). These examples were discussed thoroughly in \cite[Section 2.3]{ERS}, where also further references can be found.

\begin{example}
{\bf (Roth-Schmidt, \cite{MR0072182, MR0268129})}
Let $\xi_1,\dots,\xi_n$ be (real) algebraic numbers such that $\{\xi_1,\dots,\xi_n,1\}$ is linearly independent over $\Q$.
Then for every $\varepsilon>0$ there exists a constant $C_\varepsilon>0$ such that
    \[
    \|\alpha_1\xi_1+\cdots+\alpha_n\xi_n\|_{\R/\Z} \ge C_\varepsilon\left(\max_{1\le i \le N}|\alpha_i|\right)^{-n-\varepsilon}\]
for every $( \alpha_1, \dots, \alpha_n)\in \Z^{n}\setminus\{0\}$.
\end{example}

The second example concerns logarithms of algebraic numbers.

\begin{example}
{\bf (Baker-Feldman, \cite{Bak, Feld})}
Let $\Lambda=\{\lambda\in\C\,:\, \exp(\lambda)\in\overline{\Q}\}$, where
$\overline{\Q}$ denotes the set of all (possibly complex) algebraic numbers.
Consider
$$
\xi_1,\dots,\xi_m\in\Lambda\cap\R ,\qquad \xi_{m+1},\dots,\xi_{n}\in i\Lambda\cap \R,
$$
such that $\{\xi_1,\dots,\xi_m\}$  and $\{ \xi_{m+1},\dots,\xi_{n}\}$ are linearly independent over $\Q$.
Then there exists (effectively computable) constants $C>0$ and $\omega>0$ such that \eqref{f7.5} holds.
\end{example}

\subsection{The limit theorem for finitely generated subgroups}
\label{sec:7.2}

Let $\Xi$ be a finitely generated subgroup of $\R/\Z$ which is group-isomorphic to $\Z^n\times \Z_N$.
Then, as we will see shortly, $\HomXT$ is naturally isomorphic as a compact group to 
$\T^n\times \T_N$, where $\T_N:=\{\omega\in\C\,:\, w^N=1\}$.

We will use this observation in order to replace in the uniform version of the limit theorem (Theorem \ref{MainResult-2}) 
the continuous functions $\Theta_{A,1}$ and $\Theta_{A,2}$ defined on  $\HomXT$
by continuous functions defined on $\T^n\times \T_N$.

\medskip
Assume that $\Xi$ is group-isomorphic to $\Z^n\times \Z_N$ via the group-isomorphism \eqref{group-iso}.
Then $\HomXT$ can be identified with $\T^n\times \T_N$ as follows.
Each $(t_1,\dots,t_n,t_{n+1})\in \T^n\times \T_N$ gives rise to a $\tau\in \HomXT$ defined by
$$
\tau(\xi)=t_1^{\alpha_1}\cdots t_n^{\alpha_n} t_{n+1}^{\alpha_{n+1}}
$$
with $\xi=[\alpha_1\xi_1+\dots \alpha_n\xi_{n}+\alpha_{n+1}\frac{1}{N}]$, $(\alpha_1,\dots,\alpha_n,\alpha_{n+1})\in\Z^n\times \Z_N$. Conversely, it is easy to see,  that each $\tau\in\HomXT$ arises in this way.
In fact, the underlying map
$$
\Lambda:(t_1,\dots,t_n,t_{n+1})\mapsto \tau
$$ 
is both a group-isomorphism and a homeomorphism between $\T^n\times \T_N$ and $\HomXT$.

\medskip
The following theorem also includes the case $n=0$, in which we stipulate $\Xi=\gr\{[\frac{1}{N}]\}$ and $\beta\equiv1$.
For $n\ge1$, notice that the weight $\beta$ depends on the parameter $\omega$ given by \eqref{f7.5}, which in turn depends on $(\xi_1,\dots,\xi_n)$.

\begin{thm}[Limit theorem for finitely generated groups]\label{conclusion}
Let $n\ge0$, $(\xi_1,\dots,\xi_n)\in\mathcal{S}_n$, and $N\ge1$, and consider the group $\Xi$
with the weight $\beta$ defined by \eqref{gr.xi} and \eqref{weight.finite.group}.
Let $\mathcal{R} = \mathcal{W}_{\alpha_1,\alpha_2}(\mathcal{A})$ with $\alpha_1, \alpha_2\ge 0$, $\alpha_1 +\alpha_2 = 1$, 
and $\mathcal{A} = APW(\Z,\Xi,\beta)$.

 If $A_1,\dots, A_r\in \mathcal{R}$ and $A = e^{A_1}\cdots e^{A_r}$,
 then there exist continuous functions
$$
\widehat{\Theta}_{A,k}:\T^n\times\T_N\to\C,\qquad k=1,2,
$$
such that 
\begin{equation}\label{limit6.8b}
		\lim_{n_2 - n_1\to  \infty} \left(	  \frac{\det(P_{n_1, n_2}AP_{n_1, n_2})}{G^{n_2 - n_1}} -
		 E_{A,1}(n_1)E_{A,2}(n_2) \right) = 0
\end{equation}
with
$$
E_{A,k}(m) = \widehat{\Theta}_{A, k}
(e^{2\pi i \xi_1 m},\dots,e^{2\pi i \xi_{n} m},e^{2\pi i  \frac{1}{N} m}),\qquad k=1,2,\;\; m\in\Z,
$$
$$
G=\exp(M(a)),\quad a=D(A_1+\dots+A_r).
$$
\end{thm}
\begin{proof}
In Theorem \ref{MainResult-2} we encountered the functions $\Theta_{A,1}$ and $\Theta_{A,2}$ defined
by \eqref{Theta_1} and \eqref{Theta_2}, and we know from Proposition \ref{Theta is cont}(c) that these functions are continuous on $\HomXT$.  Considering the composition of these functions with $\Lambda$,
$$
\widehat{\Theta}_{A,k}(t_1,\dots, t_n,t_{n+1})= \Theta_{A,k}(\Lambda(t_1,\dots,t_n, t_{n+1})),
$$
we obtain continuous function  $\widehat{\Theta}_{A,k}:\T^n\times \T_N\to\C$. 
The limit \eqref{limit6.8a} contains the quantities $\Theta_{A,k}(\tau_m)$ where  $\tau_m\in \HomXT$ is given by
 $\tau_m(\xi)=e^{2\pi i m \xi}$, $\xi\in \Xi$. It is easy to see that
$$\Lambda(e^{2\pi i \xi_1 m},\dots,e^{2\pi i \xi_{n} m},e^{2\pi i \frac{1}{N} m})=\tau_m.$$
Hence
$$
\Theta_{A,k}(\tau_m)=\widehat{\Theta}_{A,k}(e^{2\pi i \xi_1 m},\dots,e^{2\pi i \xi_{n} m},e^{2\pi i \frac{1}{N} m}).
$$
We now obtain formula  \eqref{limit6.8b} from  formula \eqref{limit6.8a}.
\end{proof}

Note that we did not write down formulas for $\widehat{\Theta}_{A,k}(t_1,\dots, t_n,t_{n+1})$
in terms of operator determinants showing the dependence on $(t_1,\dots,t_{n+1})\in\T^n\times \T_N$ explicitly
since the concrete evaluation of such operator determinants seems to be illusive in general.  
Still one could ask whether such an evaluation is possible in such cases where $A$ is an almost Mathieu operator.

\medskip
While $\widehat{\Theta}_{A,k}(t_1,\dots, t_n,t_{n+1})$ are continuous, one could further ask if (or under what conditions)
these functions are differentiable in $t_1,\dots, t_n\in\T$. We leave this as an open question.

\medskip
Another issue is the vanishing of these operator determinants.
Numerical evidence suggests that for certain (exceptional) $A$ it is possible that the functions $\Theta_{A,1}(\tau)$ and $\Theta_{A,2}(\tau)$ can vanish  at particular values of $\tau$. However, an even more ``undesirable'' situation would occur if one of the functions vanishes identically on  $\HomXT$. For the case $n\ge1$ we do not know whether this can be ruled out. In the case where $n=0$ and $N\ge2$ it can happen, see Subsection \ref{sec:7.3} below.

\subsection{The case of block Laurent operators revisited}
\label{sec:7.3}

Let us consider the case of block Laurent operators again and establish a generalization of the 
strong Szeg\"o-Widom limit theorem. Such a generalization seems to be stated 
here for the first time. The classical result (in the form of  Theorem \ref{t1.1}) is recovered 
from the theorem below with $k_1=k_2$, although for a slightly different class of symbols. 

To obtain the generalization we apply Theorem \ref{MainResult-2} (or Theorem \ref{conclusion}) 
to the case of $\Xi\cong \Z_N$.
Then the operators under consideration are block Laurent operators where the symbol is a smooth
$N\times N$ matrix valued function. To be more specific, define for $\alpha_1,\alpha_2\ge 0$ the Wiener class
$W_{\alpha_1,\alpha_2}$ consisting of all $a\in L^1(\T)$ such that 
$$
\|a\|_{ W_{\alpha_1,\alpha_2}} =\sum_{k\in\Z} \alpha(k) |a_k|<\iy,
$$
where $\alpha(k)$ is defined by \eqref{alpha.k}.

\begin{thm}[Generalized block Szeg\"o-Widom limit theorem]\label{gen.block}
For $N\ge1$ and for $\alpha_1,\alpha_2\ge0$ with $\alpha_1+\alpha_2=1$, let 
$a\in W_{{\alpha_1,\alpha_2}}^{N\times N}$ be such that 
$\det a(t)\neq 0$ for all $t\in\T$ and $\det a(t)$ has winding number zero. 
Then for each $k_1,k_2\in\{0,\dots,N-1\}$ we have that 
$$
\lim_{m\to\iy} \frac{ \det P_{k_1,k_2+mN} L(a) P_{k_1,k_2+mN}}{G^{k_2-k_1+mN}}=E_{a,1}[k_1]E_{a,2}[k_2],
$$
where $G=\exp\left(\frac{1}{2\pi N}\int_0^{2\pi}\log\det a(e^{ix})\, dx\right)$ and $E_{a,1}[k_1]$ and $E_{a,2}[k_2]$ are certain constants.
\end{thm}

\begin{proof}
We first have to use a result from \cite{Eh-SzW}, more precisely, the equivalence of the statements (i) and (ii) in 
\cite[Prop.~6.4]{Eh-SzW} in the setting $\mathcal{S}=W_{\alpha_1,\alpha_2}$. Note that statement (ii) 
amounts to the above assumptions on the symbol $a$, whereas statement (i) asserts that we can write
$a=e^{a_1}\cdots e^{a_r}$ for certain $a_1,\dots,a_r\in W_{\alpha_1,\alpha_2}^{N\times N}$. 
Therefore, it follows that 
$$
L(a)=e^{L(a_1)}\cdots e^{L(a_r)}
$$
with $L(a_1),\dots,L(a_r)\in \cR$, $\cR=\mathcal{W}_{\alpha_1,\alpha_2}(APW(\Z,\Xi,\beta))$, $\Xi=\gr\{[\frac{1}{N}]\}$, $\beta\equiv1$.

We can now apply Theorem \ref{conclusion} with $n_1=k_1$, $n_2=k_2+mN$ and $m\to\iy$.
Note that  in the terms of the fractal version (Theorem \ref{MainResult-1}) we would consider the fractal sequences
$h_1(m)=k_1$ and $h_2(m)=k_2+mN$.
The expressions for the constants in \eqref{limit6.8b} are given by
$$
E_{A,j}(n_j)=\widehat{\Theta}_{A,j}(e^{2\pi i n_j\frac{1}{N}})=\widehat{\Theta}_{A,j}(e^{2\pi i k_j\frac{1}{N}})=:E_{a,j}[k_j],\qquad j=1,2.
$$
They do not depend on $m$, but only on $k_1$ and $k_2$. 
The computation of the constant $G$ can be done straightforwardly.
\end{proof}

In the classical case and when $N\ge2$  it is possible that the constant $E[a]$ in \eqref{SW-lim} is zero
(even though this is considered `exceptional'). In our situation the same can happen when $N\ge2$. 
For instance,  in the case of the symbol
$$
a(t)=\left(\begin{array}{cc}t&0\\ 0 & t^{-1}\end{array}\right)
$$
one finds that both constants $E_{a,1}[k_1]$ and $E_{a,2}[k_2]$ are zero for all $k_1,k_2$, while with little effort it can be seen that 
$a$ is a product of exponentials of trigonometric matrix functions. We leave the details to the reader.

\subsection{The inverse closedness problem}
\label{sec:7.4}

Problem (iii) mentioned at the beginning of this section asks the question how to decide whether
a given operator $A\in\cR\subseteq \cL(l^2(\Z))$ is  the finite product of exponentials of 
operators in $\cR$. It would be desirable to have at least some (non-trivial) sufficient conditions available.
We will make a connection between this problem and the question whether $\cR$ is inverse closed in $\cL(l^2(\Z))$.

Recall that a unital Banach subalgebra $\cR$ of $\cL(l^2(\Z))$ is inverse closed in $\cL(l^2(\Z))$ if
$$
\mathcal{G}(\cR) = \cR\cap \mathcal{G}(\cL(l^2(\Z))).
$$
Here $\mathcal{G}(\mathcal{B})$ stands for the group of invertible elements in a unital Banach algebra $\mathcal{B}$.

The group $\mathcal{G}(\mathcal{R})$ may consist of several connected components. 
 It is well known (see, e.g, \cite[Thm.~10.34]{rudin}) that $A$ is a finite product of exponentials  of elements in $\cR$ 
if and only if $A$ belongs to the connected component of $\mathcal{G}(\mathcal{R})$ containing the identity operator.
With this equivalence, the following proposition is obvious.

\begin{prop}
Let $\cR$ be unital, inverse closed Banach subalgebra of $\mathcal{L}(l^2(\Z))$.
For $A\in\cR$ assume that there exists a continuous function $\sigma:[0,1]\to \cR$ such that 
$$
\sigma(0)=I, \quad \sigma(1)=A,\quad\mbox{and}\quad  \sigma(t) \mbox{ is invertible in }\mathcal{L}(l^2(\Z))\mbox{ for all $t\in[0,1]$}.
$$ 
Then $A=e^{A_1}\cdots e^{A_r}$ for certain $A_1,\dots,A_r\in\cR$.
\end{prop}

Under the assumption that $\cR$ is inverse closed, this proposition provides useful sufficient criteria for $A\in \cR$ 
to be a product of exponentials. For instance, a function $\sigma$ with the required properties exists if the unbounded component of the spectrum of $A$ in $\cL(l^2(\Z))$ contains  zero. The latter is the case, for instance,  if $A$ is invertible 
in $\cL(l^2(\Z))$ and is self-adjoint.

\medskip
It would be interesting to know under which conditions $\cR$ has the inverse closedness property.
We raise the following conjecture.

\begin{conj}\label{conject}
Let $\beta$ be an admissible weight on an additive subgroup $\Xi$ of $\R/\Z$ satisfying the GRS-condition, 
and let $\alpha_1,\alpha_2\ge 0$. Then the Banach algebra 
$\cR=\mathcal{W}_{\alpha_1,\alpha_2}(APW(\Z,\Xi,\beta))$ is inverse closed in $\mathcal{L}(l^2(\Z))$.
\end{conj}

Assume that $\beta$ is an admissible weight on $\Xi$. Then the GRS-condition is necessary for the inverse
closedness of $\cR$ in  $\mathcal{L}(l^2(\Z))$. In fact, it is easy to see (e.g., by using Gelfand theory) that 
the GRS-condition is necessary and sufficient for the inverse closedness of $\cA=APW (\Z,\Xi,\beta)$ in $l^\iy(\Z)$. 
On the other hand, proving that the GRS-condition implies the inverse closedness of $\cR$ in  $\mathcal{L}(l^2(\Z))$ 
seems difficult.

In the case of finitely generated groups $\Xi\cong \Z^n$ the conjecture is motivated by the work of
Gr\"ochenig and Leinert \cite{MR3535305}, where a similar inverse closedness property is proved.
They consider (non-commutative) weighted Wiener-type algebras with twisted convolutions as the product
and symmetric weights on $\Z^d$.
Their algebras are isomorphic as Banach algebras to our algebras $\cR$ with $d=n+1$ in certain cases.
Unfortunately, the difference is that their algebras are represented on $l^2(\Z^d)$, whereas our $\cR$ is represented on $l^2(\Z)$. 

\bigskip
{\bf The almost Mathieu operators revisited.}
Recall that the almost Mathieu operator is given by
\begin{equation}\label{Mathieu-2}
M_a =  U_1 +aI+U_{-1},
\end{equation}
where $a\in AP(\Z)$ is $a(n)=\beta\cos 2\pi(\xi n +\delta)$.
We assume that $\beta$, $\xi$, and $\delta$ are real numbers in which case $M_a$ is a selfadjoint
operator on $l^2(\Z)$. 

Suppose in addition that $\xi$ is not a Liouville number. This means that $\xi$ is either rational or belongs
to the set $\mathcal{S}_1$ defined in Subsection \ref{sec:7.1}. In both cases we can conclude that there exists an
admissible and compatible weight $\beta$ on the group $\Xi_\xi=\gr\{[\xi]\}$ (see Proposition \ref{p7.6}).
In addition, this weight satisfies the GRS-condition.
It is also clear that 
$$
M_a\in\cR:= \mathcal{W}_{1/2,1/2}(APW(\Z,\Xi_\xi,\beta)).
$$

Now assume that the above conjecture is true. Then we can conclude that if $\lambda\in\C$ is such that 
$A=M_a-\lambda I$ is an invertible operator on $l^2(\Z)$, then $A$ is a finite product of exponentials 
of elements in $\cR$. Hence our limit theorems (e.g., Theorem \ref{conclusion}) are applicable
to $A$.

Without having the conjecture available we can still say something, namely, that the limit theorems
apply to $A=M_a-\lambda I$ provided that  $\lambda\in\C$ is such that 
$|\lambda|> \|M_a\|_{\cR}$.
Indeed, in this case $A$ possesses a logarithm in $\cR$. Notice that the norm of $\|M_a\|_{\cR}$ 
depends on $\beta$ and thus on the diophantine properties of $\xi$.

\end{document}